\documentclass[a4j,10pt,showkeys]{article}

\usepackage{amsmath}    
\usepackage{graphicx}   
\usepackage{geometry}
\usepackage{verbatim}   
\usepackage{color}      
\usepackage{hyperref}   
\usepackage{enumitem}

\usepackage{amssymb}
\usepackage{amsthm}
\usepackage{mathrsfs}

\newtheorem{Def}{Definition}[section]

\newtheorem{Lem}[Def]{Lemma}
\newtheorem{Thm}[Def]{Theorem}

\newtheorem{Ass}[Def]{Assumption}
\theoremstyle{definition}
\newtheorem{Rem}[Def]{Remark}
\newtheorem{Eg}[Def]{Example}

\newcommand{\rd}{\,\mathrm{d}}
\newcommand{\1}{{\bf 1}}

\newcommand{\bE}{\mathbb{E}}\newcommand{\bF}{\mathbb{F}}\newcommand{\bN}{\mathbb{N}}\newcommand{\bP}{\mathbb{P}}\newcommand{\bR}{\mathbb{R}}\newcommand{\bW}{\mathbb{W}}

\newcommand{\sF}{\mathscr{F}}

\newcommand{\bk}{\mathbf{k}}

\allowdisplaybreaks

\begin{document}
\title{
Strong solution and approximation of time-dependent radial Dunkl processes with multiplicative noise
}
\author{
Minh-Thang Do\footnote{
Institute of Mathematics, Vietnam Academy of Science and Technology, 18 Hoang Quoc Viet, Cau Giay, Hanoi, Viet Nam, email~:~\texttt{dmthang@math.ac.vn}
},
\quad
Hoang-Long Ngo\footnote{
Hanoi National University of Education, 136 Xuan Thuy, Cau Giay, Hanoi, Vietnam, email~:~\texttt{ngolong@hnue.edu.vn}
},
\quad and \quad
Dai Taguchi\footnote{
Department of Mathematics,
Kansai University,
Suita,
Osaka,
564-8680,
Japan,
email~:~\texttt{taguchi@kansai-u.ac.jp}
}
}
\maketitle
\begin{abstract}
We study the strong existence and uniqueness of solutions within a Weyl chamber for a class of time-dependent particle systems driven by multiplicative noise. This class includes well-known processes in physics and mathematical finance. We propose a method to prove the existence of negative moments for the solutions. This result allows us to analyze two numerical schemes for approximating the solutions. The first scheme is a $\theta$-Euler--Maruyama scheme, which ensures that the approximated solution remains within the Weyl chamber. The second scheme is a truncated $\theta$-Euler--Maruyama scheme, which produces values in $\bR^{d}$ instead of the Weyl chamber $\bW$, offering improved computational efficiency.
\end{abstract}

\section{Introduction}\label{Sec_1}

In this paper, we consider a stochastic particle system given by  the following stochastic differential equations (SDEs)
\begin{align}\label{SDE_1}
\rd X(t)
&
=
\sigma(t,X(t))\rd B(t)
+
b(t,X(t))\rd t
+
\sum_{\alpha \in R_{+}}
\frac{k(t,\alpha)}{\langle \alpha, X(t) \rangle}
\alpha
\rd t,~X(0) = \xi \in \bW,~t \in [0,T],
\end{align}
where $B$ is an $r$-dimensional Brownian motion defined on a  probability space $(\Omega, \sF, \bP)$ with the filtration $(\sF_{t})_{t \geq 0}$ satisfying the usual condition, $R_{+}$ is a positive root system in $\bR^{d}$, $\bW$ is a Weyl chamber in $\bR^{d}$, $b:[0,T] \times \bW \to \bR^{d}$, $\sigma:[0,T] \times \bW \to \bR^{d \times r}$ and $k:[0,T] \times R_{+}  \to (0,\infty)$ are measurable functions.

If $r = d$, $b \equiv 0$, $\sigma$ is the identity matrix, and $k$ is a multiplicity function defined on $R_{+}$, then $X = (X(t))_{t \in [0,T]}$ represents a classical radial Dunkl process. Consequently, equation \eqref{SDE_1} can be interpreted as a generalized Dunkl process with drift and multiplicative noise.

The class of radial Dunkl processes encompasses several well-known processes, including Bessel processes, the square root of multidimensional Wishart processes, and Dyson's Brownian motions.
The square root of the Bessel process, also known as the Cox--Ingersoll--Ross process in mathematical finance, and the square root of the multidimensional Wishart process serve as models for the evolution of spot interest rates in both one-dimensional and multidimensional settings, respectively.
Dyson's Brownian motions, which describe non-colliding particle systems, have been extensively studied in mathematical physics and random matrix theory, as they share the same distribution as the ordered eigenvalues of certain matrix-valued Brownian motions.

The study of radial Dunkl processes \eqref{SDE_1} is challenging due to the presence of stiff terms $\frac{1}{\langle \alpha, x \rangle}$ in the drift coefficient.
The existence and uniqueness of strong solutions for some classes of non-colliding particle systems, including Dyson's Brownian motions, have been extensively studied (see, for example, \cite{CeLe97, CeLe01,De07, GrMa13, GrMa14, NT20, NT24+, RoSh,Sc07}).
We will demonstrate the existence and uniqueness of strong solutions to equation \eqref{SDE_1}, which corresponds to a more general class of particle systems than those considered in \cite{NT20, NT24+}. It is important to note that in \cite{NT24+}, the existence of solutions to drifted Dunkl processes with additive noise was proven by using the Girsanov theorem.
However, applying that approach to equation \eqref{SDE_1} is hardly possible due to multiplicative noise.
Our new approach involves considering a sequence of truncated SDEs and showing that they converge to a process satisfying equation \eqref{SDE_1}.

The numerical approximation of SDEs whose solutions remain confined within a specific domain poses one of the most significant challenges in stochastic numerics. This challenge stems from the requirement that any approximate solution must also remain within this domain. The explicit Euler--Maruyama schemes fail to satisfy this essential property.
Many authors have studied numerical schemes for one-dimensional SDEs whose solutions remain within a subset of $\mathbb{R}$ (see \cite{Al, BeBoDi, BoDi, ChJaMi16, DeNeSz, NeSz}).
However, despite their considerable practical significance, there are relatively few results on the numerical analysis of non-colliding particle systems (\cite{LiMe, LT20, NT20, NT24+}).
To study the convergent rates of numerical scheme for SDEs whose coefficient involves singular terms of the form $\frac{1}{\langle \alpha, X(t)\rangle}$, it is crucial to show that the negative moments $\sup_{0\leq t \leq T}\bE[\langle \alpha, X(t) \rangle^{-p}]$ are finite for some $p>0$.
This result has been obtained for systems of interacting Brownian particles and Brownian particles with nearest neighbor repulsion  (see Lemma 3.4 and Lemma 3.11 in  \cite{NT20}).
Unfortunately, the bounds for negative moments in \cite{NT20} are obtained only for systems with a relatively small number of particles.
In \cite{NT24+}, utilizing the relationship between root systems and harmonic functions, the existence of negative moments has been demonstrated for random systems associated with Dunkl processes involving a large number of particles.
The approach in \cite{NT24+}, which relies on a change of probability measure, is effective for SDEs with additive noise. 

In this paper, we introduce a novel technique to establish the existence of negative moments for a broad class of non-colliding particle systems governed by SDE \eqref{SDE_1} with multiplicative noise (see Theorem \ref{Thm_negative}).
This enables us to analyze two numerical schemes for approximating the solution of SDE \eqref{SDE_1}.
The first scheme is a $\theta$-Euler--Maruyama scheme, ensuring that the approximate solution remains within the Weyl chamber $\bW$.
The second scheme is a truncated $\theta$-Euler--Maruyama scheme, which produces values in $\bR^{d}$ rather than within the Weyl chamber $\bW$, but offers the advantage of efficient computational implementation. 

This paper is organized as follows. Section \ref{Sec_2} addresses the existence, uniqueness, and moment properties of the stochastic particle system \eqref{SDE_1}.
Section \ref{Sec_3} outlines a general procedure for verifying the convergence of the $\theta$-scheme applied to SDEs with values in an open set and a one-sided Lipschitz drift.
Finally, Section \ref{Sec_4} introduces two numerical schemes for the stochastic particle system \eqref{SDE_1} and examines their convergence rates in the $L^{p}$-norm.

\subsection*{Notations}
In this paper, elements of $\bR^{d}$ are column vectors, and for $x \in \bR^{d}$, we denote $x=(x_{1},\ldots,x_{d})^{\top}$, where for a matrix $A$, $A^{\top}$ stands for the transpose of $A$. We use Frobenius norm $|A| = (\sum_{i,j}A_{i,j}^2)^{1/2}$.  
Let $\langle \cdot, \cdot \rangle$ denote the standard Euclidean inner product and $e_{1},\ldots,e_{d}$ be the standard basis vectors of $\bR^{d}$.
For a function $g: D \to \bR^d$, we denote $\|g\|_\infty = \sup_{x \in D} |g(x)|$.  For a function $f:\bR^{d} \to \bR$, we denote the gradient of $f$ by $\nabla f:=(\frac{\partial f}{\partial x_1},\ldots,\frac{\partial f}{\partial x_d})^{\top}$ and the Laplacian  by $\Delta f:=\sum_{i=1}^{d} \frac{\partial^2 f}{\partial x_i^2}$.
Let $C_b^{1,2}([0, T] \times \bW;\bR^{d})$ be the space of $\bR^{d}$-valued functions from $[0, T] \times \bW$ such that the first derivative in time and the first and second derivatives in space exist and are bounded.
For a finite set $A$, we denote $|A|$ the number of all elements of $A$.
We fix $T>0$ and a probability space $(\Omega, \sF,\bP)$ with a filtration $\bF=(\sF_{t})_{t \in [0,T]}$ satisfying the usual conditions and $B=(B(t))_{t \in [0,T]}$ is a $r$-dimensional $\bF$-Brownian motion. Whenever we define a stopping time, we use the convention that $\inf \emptyset = + \infty$. 

\section{Radial Dunkl processes with multiplicative noise}\label{Sec_2}
\subsection{Root systems, Weyl chamber and harmonicity} \label{subsec_root}
This section provides a concise overview of root systems, Weyl groups, Weyl chambers, and their relation to harmonic function. We refer the reader to \cite{DuXu, Hum12_2, Hum12_1} for a more detailed exposition.

A {\it root system} $R$ in $\bR^{d}$ is a finite set of nonzero vectors in $\bR^{d}$ such that:
(R1) \ 
$R \cap\{c \alpha ~;~c \in \bR\}=\{\alpha,-\alpha\}$ for any $\alpha \in R$;
(R2) \ 
$\sigma_{\alpha}(R)=R$ for any $\alpha \in R$, where $\sigma_{\alpha}$ is the reflection through the hyperplane perpendicular to $\alpha$. 
The element of a root system is called root.
A sub-group $W=W(R)$ of $O(d)$ is called the \emph{Weyl group} generated by a root system $R$, if it is generated by the reflections $\{\sigma_{\alpha}~;~\alpha \in R\}$.
Each root system can be written as a disjoint unions $R = R_+ \cup (-R_+)$, where $R_+$ and $-R_+$ are separated by a hyperplane $H_\beta := \{x \in \mathbb R^d : \langle \beta, x \rangle = 0\}$ with $\beta \not \in R$. Such a set $R_+$ is called a \emph{positive root system}.
A function $\bk: R \to \mathbb{R}$ is called a \emph{multiplicity function} if it is invariant under the natural action of $W$ on $R$, that is, $\bk(\alpha) = \bk(\beta)$ whenever there exists $w \in W$ such that $w \alpha = \beta$. 
A connected component of $\mathbb{R}^d \setminus \bigcup_{\alpha \in R} H_{\alpha}$ is called a \emph{Weyl chamber} of the root system $R$.
In particular, 
\begin{align*}
\bW
:=
\{x \in \bR^d~|~\langle \alpha, x \rangle>0,~\forall \alpha \in R_{+}\}
\end{align*}
is called the \emph{fundamental Weyl chamber}.

A remarkable property of (reduced) root systems in stochastic calculus is that the polynomial function $\prod_{\alpha \in R_{+}} \langle \alpha, x \rangle$, where $x \in \bW$, is $\Delta$-harmonic, meaning that $\Delta \prod_{\alpha \in R_{+}} \langle \alpha, x \rangle = 0$, (see, Theorem 6.2.6 in \cite{DuXu}).
Moreover, by using this property, for any $x \in \bW$, it holds that 
\begin{align}\label{harmonic_1}
\sum_{\alpha \in R_{+}}
\frac
{|\alpha|^2}
{\langle \alpha,x \rangle^2}
=
\sum_{\alpha, \beta \in R_{+}}
\frac
{\langle \alpha, \beta \rangle}
{\langle \alpha,x \rangle \langle \beta,x \rangle},
\end{align}
(see, e.g. Lemma 2.2 in \cite{NT24+}).

\subsection{Existence and uniqueness of radial Dunkl processes with multiplicative noise}

We define $f_{k}:[0,T] \times \bW \to \bR^{d}$ and $\overline{\sigma}:[0,T] \times \bR^{d} \to \bR$ by
\begin{align*}
f_k(t,x)
&:=
\sum_{\alpha \in R_+} \frac{k(t, \alpha)}{\langle \alpha, x \rangle}\alpha,
\\
\overline{\sigma}(t,x)
&:=
\left\{\begin{array}{ll}
\displaystyle
\max_{i=1,\ldots,d}|\sigma_{i,i}(t,x)|,
&\text{ if }
r=d
\text{ and }
\sigma_{i,j}(t,x)=0,~ \text{ for all } i \neq j,
\\
\displaystyle
|\sigma(t,x)|,
&\text{ otherwise}.
\end{array}\right.
\end{align*}
Throughout the paper, we will work under the following assumptions on the coefficients $\sigma$, $b$, and $k$.
\begin{Ass}\label{Ass_1}
\begin{enumerate}[label=(\roman*)]
\item \label{Ass_1: i}
The drift coefficient $b:[0,T] \times \bR^{d} \to \bR^{d}$ is Lipschitz continuous uniformly in time, that is, $\|b\|_{\mathrm{Lip}}:=\sup_{t \in [0,T],x \neq y}\frac{|b(t,x)-b(t,y)|}{|x-y|}<\infty$, and $\|b(\cdot,0)\|_{\infty}<\infty$.

\item \label{Ass_1: ii}
The diffusion coefficient $\sigma:[0,T] \times \bR^{d} \to \bR^{d \times r}$ is Lipschitz continuous uniformly in time, that is $\|\sigma\|_{\mathrm{Lip}}:=\sup_{t \in [0,T],x \neq y}\frac{|\sigma(t,x)-\sigma(t,y)|}{|x-y|}<\infty$.

\item \label{bounded k}
For any $(t, \alpha) \in [0,T] \times R_{+}$, it holds that $\|\overline{\sigma}(t,\cdot)\|_{\infty}^{2} \leq 2 k(t,\alpha)$ and $\|k(\cdot, \alpha)\|_{\infty}<\infty$.

\item[(iv)] \label{Ass_1: iv}
There exists $K:[0,T] \to [0,\infty)$ such that
\begin{align*}
\int_{0}^{T}K(t) \rd t<\infty
\text{ and }
\sup_{(\alpha,x) \in R_{+} \times \bW}
\frac{\langle \alpha, -b(t,x) \rangle}{\langle \alpha,x \rangle}
\leq K(t).
\end{align*}

\item[(v)] \label{Ass_1: v}
For any $(t,x) \in [0,T] \times \bW$, it holds that
\begin{align*}
\sum_{\alpha \in R_{+}}
\frac{k(t,\alpha) |\alpha|^{2}}{\langle \alpha,x \rangle^{2}}
=
\sum_{\alpha \in R_{+}}
\sum_{\beta \in R_{+}}
\frac{k(t,\beta) \langle \alpha,\beta \rangle}{\langle \alpha,x \rangle \langle \beta,x \rangle}.
\end{align*}
\end{enumerate}
\end{Ass}

\begin{Thm}\label{Thm_1}
Under Assumption \ref{Ass_1}, SDE \eqref{SDE_1} has a unique $\bW$-valued  strong solution. Moreover, for any $p \in (0,\infty)$, it holds that
\begin{align}\label{moment_1}
\bE\Big[\sup_{t \in [0,T]}|X(t)|^{p}\Big]
<\infty.
\end{align}
\end{Thm}

We consider several classes of non-colliding particle systems which satisfy Assumption \ref{Ass_1}.

\begin{Eg}[Radial Dunkl process]\label{Example_0}
Suppose that $r=d$, $\sigma$ is the identity matrix, $b=0$ and $k(t,\cdot) = \bk$  is a multiplicity function, then the solution $X$ of SDE \eqref{SDE_1} is a standard radial Dunkl process (see for more details \cite{ChDeGaRoVoYo08, De09, GaYo05, RoVo98, Sc07}).
Indeed, there exist $r \in \bN$ and irreducible root systems $R^{i}$, $i=1,\ldots,r$, such that $R^{i} \bigcap R^{j} = \emptyset$ and  $\langle \alpha^{(i)}, \alpha^{(j)} \rangle=0$, for every $\alpha^{(i)} \in R^{i}$ and $\alpha^{(j)} \in R^{j}$, $i \neq j$, and
$R = \bigcup_{j=1}^{r} R^{j}.$
Since $\bk$ is a multiplicity function, it holds that $\bk(\alpha^{j})= \bk_{j}$ for every $\alpha^{j} \in R^{j}$.
Therefore, by \eqref{harmonic_1}, it holds that
\begin{align*}
\sum_{\alpha \in R_{+}}
\sum_{\beta \in R_{+}}
\frac{\bk(\beta) \langle \alpha,\beta \rangle}{\langle \alpha,x \rangle \langle \beta,x \rangle}
&=
\sum_{j=1}^{r}
\bk_{j}
\sum_{\alpha \in R_{+}^{j}}
\sum_{\beta \in R_{+}^{j}}
\frac{\langle \alpha,\beta \rangle}{\langle \alpha,x \rangle \langle \beta,x \rangle}
=
\sum_{j=1}^{r}
\bk_{j}
\sum_{\alpha \in R_{+}^{j}}
\frac{|\alpha|^{2}}{\langle \alpha,x \rangle^{2}}
=
\sum_{\alpha \in R_{+}}
\frac{\bk(\alpha) |\alpha|^{2}}{\langle \alpha,x \rangle^{2}}.
\end{align*}
\end{Eg}

\begin{Eg}[Time-dependent Bessel process and Cox--Ingersoll--Ross model]\label{Example_1}
Suppose that $r\geq 1$, $d=1$, $R_{+}=\{1\}$, $\sigma(t,x)=\sigma(t)$, $b(t,x)=\lambda(t) x$, and $k(t,1)=k(t)$, where $\sigma:[0,T] \to \bR^{1 \times r}$, $\lambda:[0,T] \to \bR$ and $k:[0,T] \to (0,\infty)$ are measurable functions.
In this case the corresponding Weyl chamber $\bW$ is given by $\bW=(0,\infty)$.
Then the solution of equation \eqref{SDE_1} is a time-dependent Bessel process satisfying 
\begin{align}\label{Bessel_1}
\rd X(t)
=
\sigma(t) \rd B(t)
+
\lambda(t) X(t) \rd t
+
\frac{k(t)}{X(t)} \rd t.
\end{align}
We assume that $|\overline{\sigma}(t)|^{2} \leq 2 k(t)$, $\|k\|_{\infty} <\infty$ and  $\int_{0}^{T}|\lambda(t)| \rd t<\infty$, then the coefficients $\sigma$, $b$ and $k$ satisfy Assumption \ref{Ass_1} and thus SDE \eqref{Bessel_1} has a strictly positive unique strong solution.

The equation \eqref{Bessel_1} is related to the time-dependent Cox--Ingersoll--Ross model (see \cite{BeGoMi10}).
Indeed, by using It\^o's formula, the process $Y(t):=X(t)^{2}$ satisfies the equation
\begin{align*}
\rd Y(t)
=
2\sqrt{Y(t)} \sigma(t)\rd B(t)
+
\{2k(t)+|\sigma(t)|^{2}+2\lambda(t) Y(t)\}\rd t.
\end{align*}
\end{Eg}

\begin{Eg}\label{Example_2}
Suppose that $r,d>1$, $R_{+}$ is a positive root system of type $A_{d-1}$, that is, $R_{+}=\{e_{i}-e_{j}~|~i<j\}$, $k(\cdot, \alpha)=k:[0,T] \to (0,\infty)$ is a measurable function.
In this case the corresponding Weyl chamber $\bW$ is given by $\bW=\{x \in \bR^{d}~|~x_{1}>x_{2}>\cdots>x_{d}\}$.
Then $X$ is a solution of the non-colliding particle system
\begin{align}\label{Type_A}
\rd X_{i}(t)
=
\sum_{j=1}^{r}\sigma_{i,j}(t,X(t)) \rd B_{j}(t)
+
b_{i}(t,X(t)) \rd t
+
\sum_{j \neq i}
\frac{k(t)}{X_{i}(t)-X_{j}(t)}
\rd t,~i=1,\ldots,d.
\end{align}
We assume that $\sigma$ and $b$ are Lipschitz continuous in the space variable, $\|\overline{\sigma}(t,\cdot)\|_{\infty}^{2} \leq 2k(t)$, and  $b_{i} \geq b_{j}$, if $i < j$.
Then the coefficients $\sigma$, $b$ and $k$ satisfy Assumption \ref{Ass_1} with $K(t)=0$ and thus SDE \eqref{Type_A} has a unique strong solution valued in the Weyl chamber $\bW$.
\end{Eg}

\begin{Eg}
\label{Example_3}
Suppose that $r=d>1$, $R_{+}$ is a positive root system of type $B_{d}$, that is, $R_{+}=R_{+}^{1} \cup R_{+}^{2}$, $R_{+}^{1}:=\{e_{i}-e_{j}~|~i<j\} \cup \{e_{i}+e_{j}~|~i<j\}$, $R_{+}^{2}:= \{e_{i}~|~i=1,\ldots,d\}$,  $k(t,\alpha)=k_{1}(t) \1_{R_{+}^{1}}(\alpha)+k_{2}(t)\1_{R_{+}^{2}}(\alpha)$, $\sigma(t,x)=\sigma(t)$, $b(t,x)=\lambda(t)x$, where $k_{1},k_{2}:[0,T] \to (0,\infty)$, $\sigma(t):[0,T] \to \bR^{d\times r}$ and $\lambda:[0,T] \to \bR$ are measurable functions.
In this case the corresponding Weyl chamber $\bW$ is given by $\bW=\{x \in \bR^{d}~|~x_{1}>x_{2}>\cdots>x_{d}>0\}$.
Then $X$ is a solution of the non-colliding particle system
\begin{align}\label{Type_B}
\begin{split}
\rd X_{i}(t)
&=
\sum_{j=1}^{r} \sigma_{i,j}(t)\rd B_{j}(t)
+
\lambda(t) X_{i}(t) \rd t
\\&\quad+
\sum_{j \neq i}
\Big\{
\frac{k_{1}(t)}{X_{i}(t)-X_{j}(t)}
+
\frac{k_{1}(t)}{X_{i}(t)+X_{j}(t)}
\Big\}
\rd t
+
\frac{k_{2}(t)}{X_{i}(t)} \rd t
,~i=1,\ldots,d.
\end{split}
\end{align}

We assume that $|\overline{\sigma}(t)|^{2} \leq 2\min\{k_{1}(t),k_{2}(t)\}$, $\|k_{i}\|_{\infty}<\infty$, $i=1,2$ and $\int_{0}^{T}|\lambda(t)| \rd t<\infty$, then the coefficients $\sigma$, $b$ and $k$ satisfy Assumption \ref{Ass_1} and thus SDE \eqref{Type_B} has a unique strong solution values in the Weyl chamber $\bW$.
Indeed, we need to prove the condition (v) in Assumption \ref{Ass_1}.
Since $R_{+}^{1}$ and $R_{+}^{2}$ are positive root system,  by using \eqref{harmonic_1}, 
for any $(t,x) \in [0,T] \times \bW$, we have
\begin{align*}
&\sum_{\alpha \in R_{+}}
\sum_{\beta \in R_{+}}
\frac{k(t,\beta) \langle \alpha,\beta \rangle}{\langle \alpha,x \rangle \langle \beta,x \rangle}
\\&=
\sum_{\alpha \in R_{+}^{1}}
\sum_{\beta \in R_{+}^{1}}
\frac{k_{1}(t) \langle \alpha,\beta \rangle}{\langle \alpha,x \rangle \langle \beta,x \rangle}
+
\sum_{\alpha \in R_{+}^{2}}
\sum_{\beta \in R_{+}^{2}}
\frac{k_{2}(t) \langle \alpha,\beta \rangle}{\langle \alpha,x \rangle \langle \beta,x \rangle}
+
\sum_{\alpha \in R_{+}^{1}}
\sum_{\beta \in R_{+}^{2}}
\frac{k_{2}(t) \langle \alpha,\beta \rangle}{\langle \alpha,x \rangle \langle \beta,x \rangle}
+
\sum_{\alpha \in R_{+}^{2}}
\sum_{\beta \in R_{+}^{1}}
\frac{k_{1}(t) \langle \alpha,\beta \rangle}{\langle \alpha,x \rangle \langle \beta,x \rangle}
\\&=
\sum_{\alpha \in R_{+}^{1}}
\frac{k_{1}(t) |\alpha|^{2}}{\langle \alpha,x \rangle^{2}}
+
\sum_{\alpha \in R_{+}^{2}}
\frac{k_{2} |\alpha|^{2}}{\langle \alpha,x \rangle^{2}}
+
\sum_{\alpha \in R_{+}^{1}}
\sum_{\beta \in R_{+}^{2}}
\frac{k_{2}(t) \langle \alpha,\beta \rangle}{\langle \alpha,x \rangle \langle \beta,x \rangle}
+
\sum_{\alpha \in R_{+}^{2}}
\sum_{\beta \in R_{+}^{1}}
\frac{k_{1}(t) \langle \alpha,\beta \rangle}{\langle \alpha,x \rangle \langle \beta,x \rangle}
\\&=
\sum_{\alpha \in R_{+}}
\frac{k(t,\alpha) |\alpha|^{2}}{\langle \alpha,x \rangle^{2}}
+
\sum_{\alpha \in R_{+}^{1}}
\sum_{\beta \in R_{+}^{2}}
\frac{k_{2}(t) \langle \alpha,\beta \rangle}{\langle \alpha,x \rangle \langle \beta,x \rangle}
+
\sum_{\alpha \in R_{+}^{2}}
\sum_{\beta \in R_{+}^{1}}
\frac{k_{1}(t) \langle \alpha,\beta \rangle}{\langle \alpha,x \rangle \langle \beta,x \rangle}.
\end{align*}
By the definition of $R_{+}^{1}$ and $R_{+}^{2}$, it holds that
\begin{align*}
\sum_{\alpha \in R_{+}^{1}}
\sum_{\beta \in R_{+}^{2}}
\frac{k_{2}(t) \langle \alpha,\beta \rangle}{\langle \alpha,x \rangle \langle \beta,x \rangle}
&=
k_{2}(t)
\sum_{i<j}
\sum_{\ell=1}^{d}
\Big\{
\frac{\langle e_{i}-e_{j},e_{\ell} \rangle}{(x_{i}-x_{j})x_{\ell}}
+
\frac{\langle e_{i}+e_{j},e_{\ell} \rangle}{(x_{i}+x_{j})x_{\ell}}
\Big\}
\\&=
k_{2}(t)
\sum_{i<j}
\Big\{
\frac{1}{(x_{i}-x_{j})x_{i}}
-
\frac{1}{(x_{i}-x_{j})x_{j}}
+
\frac{1}{(x_{i}+x_{j})x_{i}}
+
\frac{1}{(x_{i}+x_{j})x_{j}}
\Big\}
\\&=
k_{2}(t)
\sum_{i<j}
\Big\{
\frac{-1}{x_{i}x_{j}}
+
\frac{1}{x_{i}x_{j}}
\Big\}
=0.
\end{align*}
By a similar caculation, we have $\sum_{\alpha \in R_{+}^{2}}\sum_{\beta \in R_{+}^{1}}\frac{k_{1}(t) \langle \alpha,\beta \rangle}{\langle \alpha,x \rangle \langle \beta,x \rangle}=0$.
Therefore, the condition (v) in Assumption \ref{Ass_1} holds.

The system \eqref{Type_B} is related to {\it Wishart processes}, (see e.g. \cite{Bru90,De07,GrMa13,GrMa14,KaTa04,KaTa11,KoOc01}).
Indeed, by using the equality
\begin{align*}
\sum_{j:j \neq i}
\left\{
\frac{1}{x_i-x_j}
+
\frac{1}{x_i+x_j}
\right\}
=
\frac{1}{x_{i}}
\sum_{j:j \neq i}
\left\{
\frac{x_i^2+x_j^2}{x_i^2-x_j^2}
\right\}
+
\frac{d-1}{x_{i}},
\end{align*}
and It\^o's formula, the stochastic process $Y_{i}(t):=X_{i}(t)^{2}$ satisfies the following equation:
\begin{align*}
\rd Y_{i}(t)
&=
2
\sqrt{Y_{i}(t)}
\sum_{j=1}^{r}
\sigma_{i,j}(t)
\rd B_{j}(t)
+
2\lambda(t) Y_{i}(t) \rd t
\\&\quad+
\Big\{
2k_{1}(t)(d-1)
+
2k_{2}(t)
+
\sum_{j=1}^{r}\sigma^{2}_{i,j}(s)
\Big\} \rd t
+
2k_{1}(t)
\sum_{j:j \neq i}
\frac
{Y_{i}(s)+Y_{j}(s)}
{Y_{i}(s)-Y_{j}(s)}
\rd t
,~i=1,\ldots,d.
\end{align*}
\end{Eg}

\subsection{Proof of Theorem \ref{Thm_1}}
We first introduce some auxiliary notations. 
For each  $\varepsilon>0$, we 
define $g_{\varepsilon}:\bR \to (0,\infty)$ by $g_{\varepsilon}(x):= ( \varepsilon \vee x)^{-1}$.
It is clear that
\begin{align}
&|g_{\varepsilon}(x)-g_{\varepsilon}(y)|
\leq \varepsilon^{-2}|x-y|,~\text{ for any } x,y \in \bR,
\label{g_1}
\\
&(x-y)(g_{\varepsilon}(x)-g_{\varepsilon}(y)) 
\leq 0,~ \text{ for any }  x,y \in \bR,
\label{g_2}
\\
&0 \leq x^{-1}-g_{\varepsilon}(x)
\leq
\varepsilon x^{-2},~ \text{ for any } x>0.
\label{g_3}
\end{align}
Note that the function $f_{k}$ satisfies the one-sided Lipschitz condition on $\bW$, that is, for any 
$(t,x,y)\in [0,T] \times \bW\times \bW$,
\begin{align*}
\langle x-y, f_{k}(t,x)-f_{k}(t,y) \rangle \leq 0.
\end{align*}
An $\varepsilon$-truncated version $f_{k,\varepsilon}:[0,T] \times \bR^{d} \to \bR^{d}$ of $f_{k}$ is given by 
\begin{align*}
f_{k,\varepsilon}(t,x)
:=
\sum_{\alpha \in R_{+}}
k(t,\alpha)
g_{\varepsilon}(\langle \alpha,x\rangle)
\alpha.
\end{align*}
We set
\begin{align}\label{def_L}
L_k = \sum_{\alpha \in R_+}  \|k(\cdot, \alpha)\|_{\infty} |\alpha|^2.
\end{align}
Then it follows from \eqref{g_1}, \eqref{g_2}, and \eqref{g_3} that
\begin{align}
& |f_{k,\varepsilon}(t,x)-f_{k,\varepsilon}(t,y)|
\leq
L_{k}
\varepsilon^{-2}|x-y|,~\text{ for any } t \in [0,T],~x,y \in \bR^{d},
\label{f_1}\\
&\langle x, f_{k,\varepsilon}(t,x) \rangle
\leq
\sum_{\alpha \in R_{+}} \|k(\cdot, \alpha)\|_{\infty},~\text{ for any }  (t,x) \in [0,T] \times \bR^{d},
\label{f_4}\\
&\langle x-y, f_{k,\varepsilon}(t,x)-f_{k,\varepsilon}(t,y) \rangle
\leq 0,~ \text{ for any } t \in [0,T],~x,y \in \bR^{d},
\label{f_2}\\
&|f_{k}(t,x)-f_{k,\varepsilon}(t,x)|
\leq
\varepsilon 
\sum_{\alpha \in R_{+}}
\frac{\|k(\cdot, \alpha)\|_{\infty} |\alpha|}{\langle \alpha,x \rangle^{2}},~ \text{ for any } (t,x) \in [0,T] \times \bW.
\label{f_3}
\end{align}
Under Assumption \ref{Ass_1}, by the global Lipschitz continuity of $f_{k,\varepsilon}$, 
 the following SDE  has a unique strong solution for any $\varepsilon >0$,
\begin{align}\label{SDE_app}
\rd X_{\varepsilon}(t)
=
\sigma(t,X_{\varepsilon}(t))\rd B(t)
+
b(t,X_{\varepsilon}(t))\rd t
+
f_{k,\varepsilon}(t,X_{\varepsilon}(t))
\rd t,~
X_{\varepsilon}(0)=X(0)=\xi \in \bW.
\end{align}
Note that by  \eqref{f_2}, it is easy to prove that for any $p>0$, we have
\begin{align}\label{moment_0}
\sup_{\varepsilon>0}
\bE[\sup_{t \in [0,T]}|X_{\varepsilon}(t)|^{p}]
<\infty.
\end{align}

Now we are ready to prove Theorem \ref{Thm_1}.
The proof is based on a Lyapunov function, which is inspired by the proof of Lemma 4.3.3 in \cite{AGZ10} for the case of Dyson's Brownian motion.

\begin{proof}[Proof of Theorem \ref{Thm_1}]
We set the Lyapunov function $f$ by
\begin{align*}
f(x):=|R_{+}||x|^{2}-\sum_{\alpha \in R_{+}}\log \langle \alpha,x \rangle,~x \in \bW.
\end{align*}
Then we have for any $\beta \in R_{+}$
\begin{align}\label{Lyapunov_0}
-\log \langle \beta,x\rangle
\leq
f(x)
+
\sum_{\alpha \in R_{+}}|\alpha|^{2}.
\end{align}
Indeed, if $|R_{+}|=1$, then the estimate \eqref{Lyapunov_0} is trivial.
If $|R_{+}|>1$, then it follows from  Cauchy--Schwartz's inequality and the fact that $\log |\alpha| \leq  |\alpha|^2$, that
\begin{align*}
\sum_{\alpha \neq \beta}\log \langle\alpha,x \rangle
\leq
\sum_{\alpha \neq \beta}\{\log|\alpha|+\log |x|\}
\leq
\sum_{\alpha \neq \beta}\{|\alpha|^{2}+|x|^{2}\}
\leq
|R_{+}||x|^{2}+\sum_{\alpha \in R_{+}}|\alpha|^{2}.
\end{align*}
Thus we obtain \eqref{Lyapunov_0}. For each $0<\varepsilon<\min_{\alpha \in R_{+}}\langle \alpha,\xi \rangle$, we define a stopping time $\tau_{\varepsilon}$ by
\begin{align*}
\tau_{\varepsilon}
= \inf\{t \in (0,T]~|\min_{\alpha \in R_{+}} \langle \alpha,X_{\varepsilon}(t) \rangle \leq \varepsilon\}. 
\end{align*}
Then by It\^o's formula and Assumption \ref{Ass_1}(v), we have
\begin{align*}
f(X_{\varepsilon}^{\tau_{\varepsilon}}(t))
&=
f(\xi)
+
M^{\tau_{\varepsilon}}(t)
-
\int_{0}^{t \wedge \tau_{\varepsilon}}
\sum_{\alpha \in R_{+}}
\frac{\langle \alpha,b(s,X_{\varepsilon}(s))\rangle}{\langle \alpha,X_{\varepsilon}(s)\rangle}
\rd s
\\&\quad
+
|R_{+}|
\int_{0}^{t \wedge \tau_{\varepsilon}}
\Big\{
2\sum_{\alpha \in R_{+}} k(s,\alpha)
+
2\langle X_{\varepsilon}(s),b(s,X_{\varepsilon}(s))\rangle
+
|\sigma(s,X_{\varepsilon}(s))|^{2}
\Big\}
\rd s
\\&\quad
+
\int_{0}^{t \wedge \tau_{\varepsilon}}
\sum_{\alpha \in R_{+}}
\frac{1}{\langle \alpha,X_{\varepsilon}(s) \rangle^{2}}
\Big\{
\sum_{\ell=1}^{d}
\frac{\langle \alpha,\sigma_{\cdot,\ell}(s,X_{\varepsilon}(s)) \rangle^{2}}{2}
-
k(s,\alpha)|\alpha|^{2}
\Big\}
\rd s,
\end{align*}
where $M^{\tau_{\varepsilon}}$ is a stochastic integral defined by
\begin{align*}
M^{\tau_{\varepsilon}}(t)
=
\int_{0}^{t\wedge \tau_{\varepsilon}}
\Big\langle
2|R_{+}|X_{\varepsilon}(s)
-
\sum_{\alpha \in R_{+}}
\frac{1}{\langle \alpha, X_{\varepsilon}(s)\rangle}
\alpha,
\sigma(s,X_{\varepsilon}(s))
\rd B(s)
\Big\rangle.
\end{align*}
Using Cauchy--Schwartz's inequality and Assumption \ref{Ass_1} (iii), we have
\begin{align*}
\int_{0}^{t \wedge \tau_{\varepsilon}}
\sum_{\alpha \in R_{+}}
\frac{1}{\langle \alpha,X_{\varepsilon}(s) \rangle^{2}}
\sum_{\ell=1}^{d}
\frac{\langle \alpha,\sigma_{\cdot,\ell}(s,X_{\varepsilon}(s)) \rangle^{2}}{2}
\rd s
&\leq
\int_{0}^{t \wedge \tau_{\varepsilon}}
\sum_{\alpha \in R_{+}}
\frac{\overline{\sigma}(s,x)^{2}}{2}
\frac{|\alpha|^{2}}{\langle \alpha,X_{\varepsilon}(s) \rangle^{2}}
\rd s
\\&\leq
\int_{0}^{t \wedge \tau_{\varepsilon}}
\sum_{\alpha \in R_{+}}
\frac{k(s,\alpha)|\alpha|^{2}}{\langle \alpha,X_{\varepsilon}(s) \rangle^{2}}
\rd s.
\end{align*}
Therefore, by Assumption \ref{Ass_1} (i) and (iv), we have
\begin{align*}
f(X_{\varepsilon}^{\tau_{\varepsilon}}(t))
&\leq
f(\xi)
+
M^{\tau_{\varepsilon}}(t)
+
\int_{0}^{T}
K(s)
\rd s
\\&\quad+
|R_{+}|
\int_{0}^{T}
\Big\{
2\sum_{\alpha \in R_{+}} \|k(\cdot, \alpha)\|_{\infty}
+
2| X_{\varepsilon}(s)| (\|b(\cdot,0)\|_{\infty}+\|b\|_{\mathrm{Lip}}|X_{\varepsilon}(s)|)
+
\|\sigma\|_{\infty}^{2}
\Big\}
\rd s.
\end{align*}

Let $n_0 = |f(\xi)|^{1/2} + 1$. For $n \in \bN$, $ n \geq n_0$, we set $\varepsilon_{n}:=\exp(-n^{2}-\sum_{\alpha \in R_{+}}|\alpha|^{2}) \wedge \min_{\alpha \in R_{+}}\langle \alpha,\xi \rangle/2$ and define a stopping time $\rho_{n}$ by
\begin{align*}
\rho_{n}
:=
\inf\{t \in (0,T]~|~f(X_{\varepsilon_{n}}^{\tau_{\varepsilon_{n}}}(t)) \geq n^{2}\}. 
\end{align*}
Then, since $M^{\tau_{\varepsilon_{n}}\wedge \rho_{n}}$ is a martingale, by using \eqref{moment_0}, there exists a positive constant $C_{T}>0$ which is independent of $\varepsilon$ and $n$ such that
\begin{align*}
n^{2} \bP(\rho_{n} < T)
&\leq
\bE[f(X^{\tau_{\varepsilon_{n}}\wedge \rho_{n}}(T))]
\leq
f(\xi)+C_T.
\end{align*}
This implies that $\sum_{n=n_0}^{\infty}\bP(\rho_{n} < T)<\infty$.
We set 
$\Omega_{0}:=\bigcup_{N \geq n_0} \bigcap_{n \geq N}\{\rho_{n} = T\}$.
Then by Borel--Cantelli lemma, we have $\bP(\Omega_{0})=1$.

For each $\omega \in \Omega_{0}$, there exists $N(\omega) \in \bN$ such that for any $n \geq n_0 \vee N(\omega)$, $\rho_{n}(\omega) = T$.
Then it holds that for any $\beta \in R_{+}$ and $t \in [0,T)$, $w\in \Omega_0$, $n \geq n_0 \vee N(w)$, 
\begin{align*}
-\log \langle \beta, X_{\varepsilon_{n}}^{\tau_{\varepsilon_{n}}}(t,\omega) \rangle
\leq
f(X_{\varepsilon_n}^{\tau_{\varepsilon_n}}(t,\omega))
+
\sum_{\alpha \in R_{+}}|\alpha|^{2}
<
n^{2}+\sum_{\alpha \in R_{+}}|\alpha|^{2}.
\end{align*}
This implies that, for any $t \in [0,T)$, $w\in \Omega_0$, $n \geq n_0 \vee N(w)$, 
\begin{align*}
\langle \beta, X_{\varepsilon_{n}}^{\tau_{\varepsilon_{n}}}(t,\omega) \rangle
>
\exp\Big(
-n^{2}-\sum_{\alpha \in R_{+}}|\alpha|^{2}
\Big)
\geq \varepsilon_{n}.
\end{align*}
Thus, for any $w\in \Omega_0$, $n \geq n_0 \vee N(w)$, we have  $\tau_{\varepsilon_{n}}(\omega)= T$.
Now for each $w \in \Omega$, we set $X(t,\omega):=\lim_{n \to \infty} X_{\varepsilon_{n}}(t, \omega) \1_{\Omega_{0}}(\omega)$. 
Since $X_{\varepsilon_n}(t,\omega) = X_{\varepsilon_m}(t,\omega)$ for any $m\geq n$ if $t \leq  \tau_{\varepsilon_n}(\omega)$, we have $X_{\varepsilon_n}(t,\omega) = X(t,\omega)$ whenever $t\leq  \tau_{\varepsilon_n}(\omega)$. 
It follows from \eqref{SDE_app} that 
\begin{align*}
 X_{\varepsilon_n}(t\wedge \tau_{\varepsilon_n})
=
\xi+ \int_0^{t\wedge \tau_{\varepsilon_n}} \sigma(s,X_{\varepsilon_n}(s))\rd B(s)
+
\int_0^{t\wedge \tau_{\varepsilon_n}} b(s,X_{\varepsilon_n}(s))\rd s
+
\int_0^{t\wedge \tau_{\varepsilon_n}} f_{k,\varepsilon_n}(s,X_{\varepsilon_n}(s))
\rd s.~
\end{align*}
This implies that 
\begin{align*}
 X(t\wedge \tau_{\varepsilon_n})
=
\xi + \int_0^{t\wedge \tau_{\varepsilon_n}} \sigma(s,X(s))\rd B(s)
+
\int_0^{t\wedge \tau_{\varepsilon_n}} b(s,X(s))\rd s
+
\int_0^{t\wedge \tau_{\varepsilon_n}} f_{k}(s,X(s))
\rd s.~
\end{align*}
Since  $\tau_{\varepsilon_n} \to T$ as $n \to \infty$, we conclude that  $X$ satisfies the following SDE
\begin{align*}
 X(t)
=
\xi + \int_0^{t} \sigma(s,X(s))\rd B(s)
+
\int_0^{t} b(s,X(s))\rd s
+
\int_0^{t} f_{k}(s,X(s))
\rd s,~ \text{a.s.}
\end{align*}
The estimate \eqref{moment_1} follows from  \eqref{moment_0} and the Fatou lemma. 

Finally, the pathwise uniqueness for equation \eqref{SDE_1} can be derived from the fact that $b$ and $\sigma$ are Lipschitz continuous, and $f_k$ is one-sided Lipschitz continuous. The argument is quite standard and we will omit it. 

\end{proof}


\subsection{The existence of negative moments}
We will prove that the negative moments $\sup_{0\leq t \leq T}\bE[\langle \alpha, X(t) \rangle^{-p}]$ and $\bE[ \sup_{t \in [0,T]} \langle \alpha, X(t) \rangle^{-p} ]$ are finite for some $p>0$. 
This fact plays a crucial role in the numerical approximation of $X$. 

We set $p_{*}$ by
\begin{equation*}
p_{*}
:=
\inf_{t \in [0,T],\alpha \in R_{+}}
\frac{2k(t,\alpha)}{\|\overline{\sigma}(t,\cdot)\|_{\infty}^{2}}-1.
\end{equation*}
Then by Assumption \ref{Ass_1} (iii), $p_{*} \geq 0$. In case $p_{*}=+\infty$, we set $p_{*}$ to be any constant which is greater than $6$. 


\begin{Thm}\label{Thm_negative}
Suppose that Assumptions \ref{Ass_1} holds.
If $p_{*}>0$, then for any $p \in (0,p_{*})$, it holds that
\begin{equation} \label{negative_0}
\sup_{\alpha \in R_{+}}
\sup_{t \in [0,T]}
\bE[\langle \alpha, X(t) \rangle^{-p}]
<\infty.
\end{equation}
Moreover, if $p_{*}>2$, then for any $p \in (0,p_{*}-2)$, it holds that
\begin{equation} \label{negative_1}
\sup_{\alpha \in R_{+}}
\bE
\Big[
\sup_{t \in [0,T]}
\langle \alpha, X(t) \rangle^{-p}
\Big]
<\infty.
\end{equation}
\end{Thm}
\begin{proof}
We first prove \eqref{negative_0}.
For $0 < \varepsilon<\min_{\alpha \in R_{+}}\langle \alpha,\xi \rangle$, we define a stopping time 
$$\tau_{\varepsilon}(w) 
:=\inf\{t \in (0,T]~|~\text{there exsits }\alpha \in R_{+} \text{ such that } \langle \alpha,X(t) \rangle \leq \varepsilon\}.$$
By using It\^o's formula, we have for any $\beta \in R_{+}$,
\begin{align*}
\log \langle \beta, X^{\tau_{\varepsilon}}(t) \rangle
&=
\log \langle \beta, \xi \rangle
+
\int_{0}^{t \wedge \tau_{\varepsilon}}
\frac{\langle \beta,\sigma(s,X(s)) \rd B(s) \rangle}{\langle \beta,X(s) \rangle}
+
\int_{0}^{t \wedge \tau_{\varepsilon}}
\frac{\langle \beta,b(s,X(s)) \rangle}{\langle \beta,X(s) \rangle}
\rd s
\\&\quad+
\sum_{\alpha \in R_{+}}
\int_{0}^{t \wedge \tau_{\varepsilon}}
\frac{k(s,\alpha)\langle \alpha, \beta \rangle}{\langle \alpha,X(s) \rangle \langle \beta,X(s) \rangle}
\rd s
-\frac{1}{2}
\int_{0}^{t \wedge \tau_{\varepsilon}}
\frac{1}{\langle \beta,X(s) \rangle^{2}}
\sum_{j=1}^{r}
\langle \beta,\sigma_{\cdot,j}(s,X(s)) \rangle^{2}
\rd s.
\end{align*}
By Assumption \ref{Ass_1} (iv) and (v), we have
\begin{align} \label{eqn19,5}
\log\Big(\prod_{\beta \in R_{+}} \langle \beta, X^{\tau_{\varepsilon}}(t)\rangle^{-p_{*}}\Big)
&\leq
\log\Big(\prod_{\beta \in R_{+}} \langle \beta, \xi\rangle^{-p_{*}}\Big)
-M(t)
-H(t)
+p_{*}\int_{0}^{T}K(s) \rd s,
\end{align}
where we define stochastic processes $M$ and $H$ by for $t \in [0,T]$,
\begin{align*}
M(t)&:=
p_{*}
\sum_{\beta \in R_{+}}
\int_{0}^{t \wedge \tau_{\varepsilon}}
\frac{\langle \beta,\sigma(s,X(s)) \rd B(s) \rangle}{\langle \beta,X(s) \rangle},\\
H(t)&:=
p_{*}
\int_{0}^{t \wedge \tau_{\varepsilon}}
\sum_{\beta \in R_{+}}
\frac{1}{\langle \beta,X(s) \rangle^{2}}
\Big\{
k(s,\beta)|\beta|^{2}
-
\frac{1}{2}
\sum_{j=1}^{r}
\langle \beta,\sigma_{\cdot,j}(s,X(s)) \rangle^{2}
\Big\}
\rd s.
\end{align*}

Now we prove that $\langle M\rangle(t) \leq 2 H(t)$ for any $t \in [0,T]$.
Note that by the definition of $M$ its quadratic variation $\langle M\rangle$ is given by
\begin{align*}
\langle M\rangle(t)
=
p_{*}^{2}
\int_{0}^{t \wedge \tau_{\varepsilon}}
\sum_{j=1}^{r}
\Big\langle 
\sum_{\beta \in R_{+}}
\frac{1}{\langle \beta,X(s) \rangle}
\beta,
\sigma_{\cdot,j}(s,X(s))
\Big\rangle^{2}
\rd s.
\end{align*}
Hence it is sufficient to prove that for any $(t,x) \in [0,T] \times \bW$,
\begin{align}\label{negative_2}
\mathcal{S}:= p_{*}\sum_{j=1}^{r}
\Big\langle 
\sum_{\beta \in R_{+}}
\frac{1}{\langle \beta,x \rangle}
\beta,
\sigma_{\cdot,j}(t,x))
\Big\rangle^{2}
+
\sum_{j=1}^{r}
\sum_{\beta \in R_{+}}
\frac{\langle \beta,\sigma_{\cdot,j}(t,x) \rangle^{2}}{\langle \beta,x \rangle^{2}}
\leq
\sum_{\beta \in R_{+}}
\frac{2k(t,\beta)|\beta|^{2}}{\langle \beta,x \rangle^{2}}.
\end{align}

When  $\sigma$ is a diagonal square matrix and $\bar{\sigma}(t,x)= \max_{i=1,\ldots, d}| \sigma_{i,i}(t,x)|$, 
then thanks to Assumption \ref{Ass_1} (iii) and \eqref{harmonic_1}, it holds that for any $x \in \bW$,
\begin{align*}
\mathcal{S}&\leq
p_{*}
\bar{\sigma}(t,x)^{2}
\Big\{
\sum_{j=1}^{r}
\Big(
\sum_{\beta \in R_{+}}
\frac{1}{\langle \beta,x \rangle}
\beta_{j}
\Big)^{2}
+
\sum_{\beta \in R_{+}}
\frac{|\beta|^{2}}{\langle \beta,x \rangle^{2}}
\Big\}
\\&=
\sum_{\beta \in R_{+}}
\frac{(p_{*}+1)\overline{\sigma}(t,x)|\beta|^{2}}{\langle \beta,x \rangle^{2}}
\leq
\sum_{\beta \in R_{+}}
\frac{2k(t,\beta)|\beta|^{2}}{\langle \beta,x \rangle^{2}}.
\end{align*}

On the other hand, if $\sigma$ is not a diagonal square matrix, $\bar{\sigma}(t,x) = |\sigma(t,x)|$. 
By Schwarz's inequality, Assumption \ref{Ass_1} (iii) and \eqref{harmonic_1}, it holds that for any $x \in \bW$,
\begin{align*}
\mathcal{S}&\leq
p_{*}
\Big|\sum_{\beta \in R_{+}}
\frac{1}{\langle \beta,x \rangle}
\beta
\Big|^{2}
|\sigma(t,x)|^{2}
+
\sum_{\beta \in R_{+}}
\frac{|\beta|^{2}}{\langle \beta,x \rangle^{2}}
|\sigma(t,x)|^{2}
\\&=
\sum_{\beta \in R_{+}}
\frac{(p_{*}+1)|\sigma(t,x)|^{2}|\beta|^{2}}{\langle \beta,x \rangle^{2}}
\leq
\sum_{\beta \in R_{+}}
\frac{2k(t,\beta)|\beta|^{2}}{\langle \beta,x \rangle^{2}}.
\end{align*}

This proves \eqref{negative_2}, and thus $\langle M\rangle(t)  \leq 2 H(t)$ for any $t \in [0,T]$.

By the definition of the stopping time $\tau_{\varepsilon}$,  $(\exp(-M(t)-\langle M\rangle(t)/2))_{t \in [0,T]}$ is a martingale. Thus, thanks to \eqref{eqn19,5}, we have
\begin{align*}
&\bE\Big[\prod_{\beta \in R_{+}} \langle \beta, X^{\tau_{\varepsilon}}(t) \rangle^{-p_{*}}\Big]
\\&\leq
\prod_{\beta \in R_{+}} \langle \beta, \xi \rangle^{-p_{*}}
\exp\Big(
p_{*}
\int_{0}^{T}
K(s)
\rd s
\Big)
\bE[\exp(-M(t)-\langle M\rangle(t)/2)\exp(\langle M\rangle(t)/2-H(t))]
\\&\leq
\prod_{\beta \in R_{+}} \langle \beta, \xi \rangle^{-p_{*}}
\exp\Big(
p_{*}
\int_{0}^{T}
K(s)
\rd s
\Big)
\bE[\exp(-M(t)-\langle M\rangle(t)/2)]
\\&=
\prod_{\beta \in R_{+}} \langle \beta, \xi \rangle^{-p_{*}}
\exp\Big(
p_{*}
\int_{0}^{T}
K(s)
\rd s
\Big).
\end{align*}
Hence by Fatou's lemma, we have
\begin{align*}
\sup_{t \in [0,T]}
\bE\Big[\prod_{\beta \in R_{+}} \langle \beta, X(t) \rangle^{-p_{*}}\Big]
\leq
\prod_{\beta \in R_{+}} \langle \beta, \xi \rangle^{-p_{*}}
\exp\Big(
p_{*}
\int_{0}^{T}
K(s)
\rd s
\Big).
\end{align*}
Finally, by using H\"older's inequality and \eqref{moment_1}, for any $p \in (0,p_{*})$, it holds that,
\begin{align*}
&
\sup_{\alpha \in R_{+}}
\sup_{t \in [0,T]}
\bE\Big[
\langle \alpha, X(t) \rangle^{-p}
\Big]
=
\sup_{\alpha \in R_{+}}
\sup_{t \in [0,T]}
\bE\Big[
\prod_{\beta \in R_{+}}
\langle \beta, X(t) \rangle^{-p}
\prod_{\beta \neq \alpha}
\langle \beta, X(t) \rangle^{p}
\Big]
\\&\leq
\sup_{\alpha \in R_{+}}
\sup_{t \in [0,T]}
\bE\Big[
\prod_{\beta \in R_{+}}
\langle \beta, X(t) \rangle^{-p_{*}}
\Big]^{p/p_{*}}
\bE\Big[
\prod_{\beta \neq \alpha}
(|\beta|| X(t)|)^{\frac{p p_{*}}{p_{*}-p}}
\Big]^{(p_{*}-p)/p_{*}}
<\infty.
\end{align*}
This concludes the proof of \eqref{negative_0}.

Now we prove \eqref{negative_1}.
Let $X_{\alpha}(t)=\langle \alpha,X(t) \rangle$ and $Y_{\alpha}(t)=X_{\alpha}(t)^{-1}$, $t \in [0,T]$ and $\alpha \in R_{+}$.
Then $X_{\alpha}$ satisfies the following equation
\begin{align*}
X_{\alpha}(t)
=
\langle \alpha, \xi \rangle
+
\int_{0}^{t}
\langle \alpha, \sigma(s,X(s)) \rd B(s) \rangle
+
\int_{0}^{t}
\langle \alpha, b(s,X(s)) \rangle
\rd s
+
\sum_{\beta \in R_{+}}
\int_{0}^{t}
\frac{k(s,\beta)\langle \alpha,\beta \rangle}{X_{\beta}(s)}
\rd s.
\end{align*}
By using It\^o's formula, we have
\begin{align*}
Y_{\alpha}(t)^{p}
&=
\langle \alpha, \xi \rangle^{-p}
-p
\int_{0}^{t}
Y_{\alpha}(s)^{p+1}
\langle \alpha, \sigma(s,X(s))\rd B(s) \rangle
-p
\int_{0}^{t}
Y_{\alpha}(s)^{p+1}
\langle \alpha, b(s,X(s)) \rangle
\rd s
\\&\quad
-p
\sum_{\beta \in R_{+}}
\int_{0}^{t}
k(s,\beta)
\langle \alpha,\beta \rangle
Y_{\alpha}(s)^{p+1}
Y_{\beta}(s)
\rd s
+
\frac{p(p+1)}{2}
\sum_{j=1}^{r}
\int_{0}^{t}
Y_{\alpha}(s)^{p+2}
\langle \alpha,\sigma_{\cdot,j}(s,X(s)) \rangle^{2}
\rd s.
\end{align*}
Let $p \in (0,p_{*}-2)$.
Note that by Assumption \ref{Ass_1}, $\sigma$ and $k$ are bounded, and $b$ is of linear growth.
Hence, by using Burkholder-Davis-Gundy's inequality, Schwarz's inequality, and Young's inequality,
there exist positive constants  $C_{1}(p), C_2(p)$ such that
\begin{align*}
&\bE\Big[
\sup_{0 \leq t \leq T}
Y_{\alpha}^{\tau_{\varepsilon}}(t)^{p}
\Big]
\\&\leq
C_1(p)
\Big\{
\langle \alpha,\xi \rangle^{-p}
+
\bE\Big[
\Big(
\int_{0}^{T \wedge \tau_{\varepsilon}}
Y_{\alpha}(s)^{2p+2}
\rd s
\Big)^{1/2}
\Big]
+
\sum_{\beta \in R_{+}}
\int_{0}^{T}
\bE[Y_{\alpha}(s)^{p+1}Y_{\beta}(s)]
\rd s
\\&\hspace{6.9cm}\quad
+
\int_{0}^{T}
\bE[Y_{\alpha}(s)^{p+1}(1+|X(s)|^{p})+Y_{\alpha}(s)^{p+2}]
\rd s
\Big\}
\\&\leq
C_2(p)
\Big\{
\langle \alpha,\xi \rangle^{-p}
+
\bE\Big[
\sup_{0 \leq t \leq T}
Y_{\alpha}^{\tau_{\varepsilon}}(t)^{p/2}
\Big(
\int_{0}^{T \wedge \tau_{\varepsilon}}
Y_{\alpha}(s)^{p+2}
\rd s
\Big)^{1/2}
\Big]
+
\sum_{\beta \in R_{+}}
\int_{0}^{T}
\bE[Y_{\beta}(s)^{p+2}]
\rd s
\\&\hspace{5cm}\quad
+
\int_{0}^{T}
\bE[Y_{\alpha}(s)^{p+1}+Y_{\alpha}(s)^{p+2}+|X(s)|^{p(p+2)}]
\rd s
\Big\}.
\end{align*}
By using Young's inequality, it holds that
\begin{align*}
C_2(p)\bE\Big[
\sup_{0 \leq t \leq T}
Y_{\alpha}^{\tau_{\varepsilon}}(t)^{p/2}
\Big(
\int_{0}^{T \wedge \tau_{\varepsilon}}
Y_{\alpha}(s)^{p+2}
\rd s
\Big)^{1/2}
\Big]
\leq
\frac{1}{2}
\bE\Big[
\sup_{0 \leq t \leq T}
Y_{\alpha}^{\tau_{\varepsilon}}(t)^{p}
\Big]
+
\frac{C_2(p)^2}{2}
\bE\Big[
\int_{0}^{T}
Y_{\alpha}(s)^{p+2}
\rd s
\Big].
\end{align*}
Hence, by using \eqref{moment_1} and \eqref{negative_0}, there exist a positive constant  $C_3(p)$ such that
\begin{align*}
\bE\Big[
\sup_{0 \leq t \leq T}
Y_{\alpha}^{\tau_{\varepsilon}}(t)^{p}
\Big]
&\leq
C_3(p)
\Big\{
\langle \alpha,\xi \rangle^{-p}
+
\sum_{\beta \in R_{+}}
\int_{0}^{T}
\bE[Y_{\beta}(s)^{p+2}]
\rd s
\\&\hspace{2cm}\quad+
\int_{0}^{T}
\bE[Y_{\alpha}(s)^{p+1}+Y_{\alpha}(s)^{p+2}+|X(t)|^{p(p+2)}]
\rd s
\Big\}.
\end{align*}
By Fatou's lemma, we have $\sup_{\alpha \in \bR_{+}}\bE[\sup_{0 \leq t \leq T}Y_{\alpha}(t)^{p}]<\infty$.
This concludes the proof of \eqref{negative_1}.
\end{proof}

As an application of the estimate of the negative moment, we get the following Kolmogorov-type condition.

\begin{Lem}\label{Lem_Kol_0}
Suppose that Assumptions \ref{Ass_1} holds and $p_{*}>1$, then for any $p \in [1,p_{*})$, there exists $C_{p}>0$ such that for any $t,s \in [0,T]$,
\begin{align*}
\bE[|X(t)-X(s)|^{p}]
\leq
C_{p}
|t-s|^{p/2}.
\end{align*}
\end{Lem}
\begin{proof}
Then since $\sigma$ and $k$ are bounded, and $b$ is of linear growth, by using Burkholder--Davis--Gundy's inequality, \eqref{moment_0} and \eqref{negative_0}, there exits $C_{p>0}$ such that for any $t,s \in [0,T]$ with $s<t$,
\begin{align*}
&\bE[|X(t)-X(s)|^{p}]\\
&\leq
C_{p}
\Big\{
\bE\Big[
\Big(
\int_{s}^{t}
|\sigma(u,X(u))|^{2}
\rd s
\Big)^{p/2}
\Big]
+
|t-s|^{p-1}
\int_{s}^{t}
\bE\Big[
1+|X(u)|^{p}
+\sum_{\alpha \in R_{+}}
\frac{1}{\langle \alpha, X(u) \rangle^p}
\Big]
\rd u
\Big\}
\\&\leq
C_{p}(t-s)^{p/2}.
\end{align*}
This concludes the assertion.
\end{proof}

\section{A general theory on $\theta$-scheme for SDEs with one-sided Lipschitz drift}\label{Sec_3}

In this section, we propose a general procedure to verify the convergence of $\theta$-scheme applied for SDEs valued in an open set with one-sided Lipschitz drift.


Let $T\in (0,\infty)$, $d,r\in \bN$, let $D \subset \bR^{d}$ be open set, and let $V=(V(t))_{t \in [0,T]}$ be a solution of the SDEs of the form
\begin{equation*}
 \rd V(t)
=
\sigma(t,V(t))\rd B(t)
+
b(t,V(t))
\rd t
+
g(t,V(t))
\rd t,~t \in [0,T],~
V(0)=\xi \in D,
\end{equation*}
where $\sigma :[0,T] \times\bR^{d}\to\bR^{d\times r}$, $b:[0,T]\times\bR^d\to\bR^d$ and $g: [0,T]\times \bR^d\to \bR^d$ are mesurable functions.

We need the following assumptions for $V$ and the coefficients $\sigma$, $b$ and $g$.

\begin{Ass}\label{Ass_2}
\begin{enumerate}[label=(\roman*)]
\item It holds that
$\bP(V(t)\in D,~\forall t \in [0,T])=1$.
\item
There exists $K>0$ which satisfies for all $(t,x) \in [0,T] \times D$ that $\langle x,g(t,x) \rangle \leq K$.
Moreover, it holds for all $(t,x,y) \in [0,T] \times D^{2}$ that $\langle x-y,g(t,x)-g(t,y) \rangle \leq 0$.
\item 
There exists $L_{b,\sigma}\in (0,\infty)$ which satisfies that
\begin{equation*}
\sup_{f\in\{b,\sigma\}}\sup_{\substack{s,t\in [0,T]\\ s\neq t}}\sup_{\substack{x,y\in \bR^d\\x\neq y}} \bigg(\frac{|f(t,x)-f(s,x)|}{|t-s|^{1/2}}+\frac{|f(t,x)-f(t,y)|}{|x-y|}+|\sigma(t,x)|\bigg)\leq L_{b,\sigma}.
\end{equation*}
\item
There exists $h_0 \in(0,\infty)$ such that for all $(t,x) \in [0,T] \times \bR^{d}$ and $h \in (0,h_0)$, the system of equation $y=x+hg(t,y)$ has a unique solution in $D$.
\end{enumerate}
\end{Ass}




Let $n \in \bN, n \geq Th_0^{-1}$, and $\theta \in [0, 1/2)$, and we set $\Delta t:=T/n$, $\Delta B_{\ell}:=B(t_{\ell+1})-B(t_{\ell})$ and $t_{\ell}:=\ell \Delta t$ for each $\ell =0,\ldots,n$.
Under Assumption \ref{Ass_2}, we define the $\theta$-scheme $V^{(n,\theta)}=(V^{(n,\theta)}(t_{\ell}))_{\ell=0}^{n}$ for $V$ as follows:
$V^{(n,\theta)}(0):=V(0)=\xi\in D$ and for each $\ell=0,1,\ldots,n-1$, $V^{(n)}(t_{\ell+1})$ is the unique solution in $D$ of the following equation:
\begin{equation*}
\begin{split}
V^{(n,\theta)}(t_{\ell+1})
&=
V^{(n,\theta)}(t_{\ell})
+
\sigma(t_{\ell},V^{(n,\theta)}(t_{\ell}))
\Delta B_{\ell}
+
b(t_{\ell},V^{(n,\theta)}(t_{\ell}))
\Delta t
\\&\quad+
\theta
g(t_{\ell},V^{(n,\theta)}(t_{\ell}))
\Delta t
+
(1-\theta)
g(t_{\ell+1},V^{(n,\theta)}(t_{\ell+1}))
\Delta t.
\end{split}
\end{equation*}

We define the approximation error $e^{(n,\theta)}(\ell)$ for $\ell=0,\ldots,n$ by
\begin{align*}
e^{(n,\theta)}(\ell):=V(t_{\ell})-V^{(n,\theta)}(t_{\ell}).
\end{align*}
Then we use the following representation for $e^{(n,\theta)}(\ell+1)$:
\begin{align}\label{property_e}
e^{(n,\theta)}(\ell+1)
&=
e^{(n,\theta)}(\ell)
+
R^{(n,\theta)}(\sigma,b,\ell)
+
\theta R^{(n,\theta)}(g,\ell)+(1-\theta) R^{(n,\theta)}(g,\ell+1)
+
r^{(n)}(\ell),
\end{align}
where $R^{(n,\theta)}(\sigma,b,\ell):=R^{(n,\theta)}(\sigma,\ell)+R^{(n,\theta)}(b,\ell)$, $r^{(n)}(\ell):=r^{(n)}(\sigma,\ell)+r^{(n)}(b,\ell)+r^{(n)}(g,\ell)$ and
\begin{align*}
R^{(n,\theta)}(\sigma,\ell)&:=
\{
\sigma(t_{\ell},V(t_{\ell}))
-
\sigma(t_{\ell},V^{(n,\theta)}(t_{\ell}))
\}
\Delta B_{\ell},
\\
R^{(n,\theta)}(b,\ell)
&:=
\{
b(t_{\ell},V(t_{\ell}))
-
b(t_{\ell},V^{(n,\theta)}(t_{\ell}))
\}
\Delta t,
\\
R^{(n,\theta)}(g,\ell)
&:=
\{
g(t_{\ell},V(t_{\ell}))
-
g(t_{\ell},V^{(n,\theta)}(t_{\ell}))
\}
\Delta t,\\
r^{(n)}(\sigma,\ell)
&:=
\int_{t_{\ell}}^{t_{\ell+1}}
\big\{
\sigma(s,V(s))
-
\sigma(t_{\ell},V(t_{\ell}))
\big\}
\rd B(s),
\\
r^{(n)}(b,\ell)
&:=
\int_{t_{\ell}}^{t_{\ell+1}}
\big\{
b(s,V(s))
-
b(t_{\ell},V(t_{\ell}))
\big\}
\rd s,
\\
r^{(n)}(g,\ell)
&:=
\int_{t_{\ell}}^{t_{\ell+1}}
\big\{
g(s,V(s))
-
\theta
g(t_{\ell},V(t_{\ell}))
-
(1-\theta)
g(t_{\ell+1},V(t_{\ell+1}))
\big\}
\rd s.
\end{align*}

We obtain the following estimate for the $\theta$-scheme $V^{(n,\theta)}$.

\begin{Thm}\label{main_3}
Suppose that Assumption \ref{Ass_2} holds.
Assume that there exists $C>0$ such that for all $n \in \bN, n \geq Th_0^{-1},$ it holds that
\begin{align}\label{ass: main_3}
\sup_{\ell=1,\ldots,n}\bE[|r^{(n)}(g,\ell)|^{2}]
\leq
C(\Delta t)^{3}\quad\text{and}\quad \sup_{\ell=0,1,\dots,n-1}\sup_{s\in[t_\ell,t_{\ell+1}]}\bE[|V(s)-V(t_\ell)|^2]\leq \frac{C}{n}.
\end{align}
Then for any $\theta_{0}\in [0,1/2)$, there exists $C_{\theta_{0}}>0$ such that for all $n \in \bN$ it holds that
\begin{align*}
\sup_{\theta \in [0,\theta_{0}]}
\bE\Big[\sup_{\ell =1,\ldots,n}\Big|V(t_{\ell})-V^{(n,\theta)}(t_{\ell})\Big|^{2}\Big]^{1/2}
\leq
\frac{C_{\theta_{0}}}{\sqrt{n}}.
\end{align*}
\end{Thm}

\begin{Rem}
Higham, Mao, and Stuart \cite{HMS} studied a backward Euler--Maruyama scheme. Under the assumptions of a one-sided Lipschitz condition and local Lipschitz continuity for the drift coefficient, they showed that the scheme converges in the $L^{2}$-sup norm with an order of $1/2$ (see Theorem 5.3 in \cite{HMS}). However, in our context, the drift coefficient $g$ is not assumed to satisfy the local Lipschitz condition. Therefore, their result does not directly lead to the estimate presented in Theorem \ref{main_3}.
\end{Rem}

\begin{proof}[Proof of Theorem \ref{main_3}]
We first prove
\begin{align}\label{main_3_0}
\sup_{\theta \in [0,\theta_{0}]}
\sup_{\ell =1,\ldots,n}
\bE\Big[\Big|V(t_{\ell})-V^{(n,\theta)}(t_{\ell})\Big|^{2}\Big]^{1/2}
\leq
\frac{C_{\theta_{0}}}{\sqrt{n}}.
\end{align}
Let $\theta \in [0,\theta_{0}]$ be fixed.
By using the representation \eqref{property_e}, for each $\ell\in 0,1,\dots,n-1$, we have
\begin{equation*}
|e^{(n,\theta)}(\ell+1)-(1-\theta)R^{(n,\theta)}(g,\ell+1)|^{2}
=
|e^{(n,\theta)}(\ell)+\theta R^{(n,\theta)}(g,\ell)+R^{(n,\theta)}(\sigma,b,\ell)+r^{(n)}(\ell)|^{2}.
\end{equation*}
By Assumption \ref{Ass_2} (ii), it holds that $\langle e^{(n,\theta)}(\ell),R^{(n,\theta)}(g,\ell) \rangle \leq 0$.
Hence we obtain
\begin{equation}\label{eqn23.5}
\begin{split}
|e^{(n,\theta)}(\ell+1)|^{2}
&\leq
|e^{(n,\theta)}(\ell)|^{2}
+
2\langle e^{(n,\theta)}(\ell), R^{(n,\theta)}(\sigma,\ell)+R^{(n,\theta)}(b,\ell)+r^{(n)}(\ell)\rangle
\\&\quad
+2\theta
\langle R^{(n,\theta)}(g,\ell), R^{(n,\theta)}(\sigma,b,\ell)+r^{(n)}(\ell)\rangle
+
|R^{(n,\theta)}(\sigma,b,\ell)+r^{(n)}(\ell)|^{2}
\\&\quad+
\theta^{2} |R^{(n,\theta)}(g,\ell)|^{2}
-
(1-\theta)^{2}|R^{(n,\theta)}(g,\ell+1)|^{2}.
\end{split}
\end{equation}
By using Young's inequality, $2 \theta |\langle a,b\rangle| \leq (1-2\theta) |a|^{2}+\frac{\theta^2}{1-2\theta}|b|^{2}$ for $a,b \in \bR^{d}$, we have
\begin{equation*}
2\theta
\langle R^{(n,\theta)}(g,\ell), R^{(n,\theta)}(b,\ell)+r^{(n)}(\ell)\rangle
\leq
(1-2\theta)|R^{(n,\theta)}(g,\ell)|^{2}
+
\frac{\theta^{2}}{1-2\theta}|R^{(n,\theta)}(\sigma,b,\ell)+r^{(n)}(\ell)|^{2}.
\end{equation*}
Note that  $\theta^{2}+(1-2\theta)=(1-\theta)^{2}$.
It follows from \eqref{eqn23.5} and the fact that $\theta \in [0,\theta_0]$  that
\begin{align*}
|e^{(n,\theta)}(\ell+1)|^{2}
&\leq
|e^{(n,\theta)}(\ell)|^{2}
+
2\langle e^{(n,\theta)}(\ell), R^{(n,\theta)}(\sigma,\ell)+R^{(n,\theta)}(b,\ell)+r^{(n)}(\ell)\rangle
\\&\quad+
\Big\{1+\frac{\theta_{0}^{2}}{1-2\theta_{0}}\Big\}
|R^{(n,\theta)}(\sigma,b,\ell)+r^{(n)}(\ell)|^{2}
\\&\quad+
(1-\theta)^{2}\big\{|R^{(n,\theta)}(g,\ell)|^{2}-|R^{(n,\theta)}(g,\ell+1)|^{2}\big\}.
\end{align*}
We consider the term $\langle e^{(n,\theta)}(\ell), R^{(n,\theta)}(b,\ell)\rangle$.
By using Cauchy--Schwarz inequality and Assumption \ref{Ass_2} (iii), we have
\begin{align*}
\langle e^{(n,\theta)}(\ell), R^{(n,\theta)}(b,\ell)\rangle
&=\langle e^{(n,\theta)}(\ell),
b(t_{\ell},V(t_{\ell}))
-
b(t_{\ell},V^{(n,\theta)}(t_{\ell}))
\rangle
\Delta t
\leq
L_{b,\sigma}|e^{(n,\theta)}(\ell)|^{2} \Delta t.
\end{align*}
Therefore, we obtain
\begin{align}\label{eq_914_0}
\begin{split}
|e^{(n,\theta)}(\ell+1)|^{2}
&\leq
|e^{(n,\theta)}(\ell)|^{2}
+
L_{b,\sigma}|e^{(n,\theta)}(\ell)|^{2} \Delta t
+
2\langle e^{(n,\theta)}(\ell), R^{(n,\theta)}(\sigma,\ell)+r^{(n)}(\ell)\rangle
\\&\quad+
\Big\{1+\frac{\theta_{0}^{2}}{1-2\theta_{0}}\Big\}
|R^{(n,\theta)}(\sigma,b,\ell)+r^{(n)}(\ell)|^{2}
\\&\quad+
(1-\theta)^{2}\big\{|R^{(n,\theta)}(g,\ell)|^{2}-|R^{(n,\theta)}(g,\ell+1)|^{2}\big\}.
\end{split}
\end{align}
Since $\sigma$ is bounded, it holds that
\begin{align*}
\bE[\langle e^{(n,\theta)}(\ell), R^{(n,\theta)}(\sigma,\ell)\rangle]
&=
\bE[\langle e^{(n,\theta)}(\ell), \{\sigma(t_{\ell},V(t_{\ell}))-\sigma(t_{\ell},V^{n,\theta}(t_{\ell}))\} \bE[\Delta B_{\ell}|\sF_{t_{\ell}}] \rangle]
=
0,\\
\bE[\langle e^{(n,\theta)}(\ell), r^{(n)}(\sigma,\ell)\rangle]
&=
\bE\Big[
\Big\langle e^{(n,\theta)}(\ell),
\bE\Big[
\int_{t_{\ell}}^{t_{\ell+1}}
\big\{
\sigma(s,V(s))
-
\sigma(t_{\ell},V(t_{\ell}))
\big\}
\rd B(s)
\Big|\sF_{t_{\ell}}
\Big]
\Big\rangle
\Big]
=0.
\end{align*}
Therefore by taking the expectation in the both sides in \eqref{eq_914_0}, we have
\begin{equation}\label{eq: e3}
\begin{split}
\bE[|e^{(n,\theta)}(\ell+1)|^{2}]
&\leq
\bE[|e^{(n,\theta)}(\ell)|^{2}]
+
2L_{b,\sigma}\bE[|e^{(n,\theta)}(\ell)|^{2}] \Delta t
+
2\bE[\langle e^{(n,\theta)}(\ell),r^{(n)}(\ell)\rangle]
\\&\quad+
\Big\{1+\frac{\theta_{0}^{2}}{1-2\theta_{0}}\Big\}
\bE[|R^{(n,\theta)}(\sigma,b,\ell)+r^{(n)}(\ell)|^{2}]
\\&\quad+
(1-\theta)^{2}
\bE[|R^{(n,\theta)}(g,\ell)|^{2}-|R^{(n,\theta)}(g,\ell+1)|^{2}].
\end{split}
\end{equation}

Next we consider the estimate for $\bE[|R^{(n,\theta)}(\sigma,b,\ell)+r^{(n)}(\ell)|^{2}]$.
By the definition of $R^{(n,\theta)}(\sigma,\ell)$, $R^{(n,\theta)}(b,\ell)$, $R^{(n,\theta)}(b,\ell)$, the assumption of $r^{(n)}(g,\ell)$ and Assumption \ref{Ass_2} (iii), there exists a constant $C_{1}>0$ such that
\begin{align}
\bE[|R^{(n,\theta)}(\sigma,\ell)|^{2}]
&\leq
C_{1}\bE[|e^{(n,\theta)}(\ell)|^{2}] \Delta t,
\label{R_esti_0}\\
\bE[|R^{(n,\theta)}(b,\ell)|^{2}]
&\leq
C_{1}\bE[|e^{(n,\theta)}(\ell)|^{2}] (\Delta t)^{2},
\label{R_esti_1}\\
\bE[|r^{(n)}(b,\ell)|^{2}]
&\leq
C_{1}(\Delta t)^{3}\label{R_esti_2}.
\end{align}
Moreover, by using It\^o's isometry, Assumption \ref{Ass_2} (iii), and \eqref{ass: main_3}, there exists $C_{2}>0$ such that
\begin{align}\label{theta scheme: eq9}
\bE[|r^{(n)}(\sigma,\ell)|^{2}]
&=
\sum_{j=1}^{r}
\bE\Big[
\Big|
\int_{t_{\ell}}^{t_{\ell+1}}
\sum_{i=1}^{d}
\big\{
\sigma_{i,j}(s,V(s))
-
\sigma_{i,j}(t_{\ell},V(t_{\ell}))
\big\}
\rd B_{j}(s)
\Big|^{2}
\Big]
\notag
\\&=
\sum_{j=1}^{r}
\bE\Big[
\int_{t_{\ell}}^{t_{\ell+1}}
\Big|
\sum_{i=1}^{d}
\big\{
\sigma_{i,j}(s,V(s))
-
\sigma_{i,j}(t_{\ell},V(t_{\ell}))
\big\}
\Big|^{2}
\rd s
\Big]
\notag
\\&\leq
d
\bE\Big[
\int_{t_{\ell}}^{t_{\ell+1}}
\big|
\sigma(s,V(s))
-
\sigma(t_{\ell},V(t_{\ell}))
\big|^{2}
\rd s
\Big]
\leq
C_{2}(\Delta t)^{2}.
\end{align}
By using Young's inequality, assumption \eqref{ass: main_3}, and estimate \eqref{R_esti_2}, we obtain
\begin{align}
\bE[\langle e^{(n,\theta)}(\ell), r^{(n)}(b,\ell)\rangle]
&\leq
\frac{\Delta t}{2}
\bE[|e^{(n,\theta)}(\ell)|^{2}]
+
\frac{(\Delta t)^{-1}}{2}
\bE[r^{(n)}(b,\ell)^{2}]
\notag\\&\leq
\frac{\Delta t}{2}
\bE[|e^{(n,\theta)}(\ell)|^{2}]
+
\frac{C_{1}}{2}(\Delta t)^{2}, \label{theta scheme: eq10}\\
\bE[\langle e^{(n,\theta)}(\ell), r^{(n)}(g,\ell)\rangle]
&\leq
\frac{\Delta t}{2}
\bE[|e^{(n,\theta)}(\ell)|^{2}]
+
\frac{(\Delta t)^{-1}}{2}
\bE[r^{(n)}(g,\ell)^{2}]
\notag\\&\leq
\frac{\Delta t}{2}
\bE[|e^{(n,\theta)}(\ell)|^{2}]
+
\frac{C_{g}}{2}(\Delta t)^{2} \label{theta scheme: eq11}.
\end{align}
Combining \eqref{ass: main_3}, \eqref{eq: e3}, \eqref{R_esti_0}, \eqref{theta scheme: eq9}, \eqref{theta scheme: eq10}, \eqref{theta scheme: eq11}, there exists $C_{3}>0$ such that
\begin{align*}
\bE[|e^{(n,\theta)}(\ell+1)|^{2}]
&\leq
\bE[|e^{(n,\theta)}(\ell)|^{2}]
+
C_{3}\bE[|e^{(n,\theta)}(\ell)|^{2}] \Delta t
+
C_{3} (\Delta t)^{2}
\notag\\&+
(1-\theta)^{2}
\bE[|R^{(n,\theta)}(g,\ell)|^{2}-|R^{(n,\theta)}(g,\ell+1)|^{2}].
\end{align*}
Therefore, by induction with the fact that $R^{(n,\theta)}(g,0)=0$, we obtain
\begin{align*}
\bE[|e^{(n,\theta)}(\ell+1)|^{2}]
&\leq C_{3}\Delta t\sum_{i=0}^{\ell}\bE[|e^{(n,\theta)}(i)|^{2}] +C_{3}T(\Delta t)^2
\notag\\&\quad+
(1-\theta)^2 \sum_{i=0}^{\ell}\bE[|R^{(n,\theta)}(g,i)|^{2}-|R^{(n,\theta)}(g,i+1)|^{2}]\\
&\leq C_{3}\Delta t\sum_{i=0}^{\ell} \bE[|e^{(n,\theta)}(i)|^{2}] +C_{3}T \Delta t.
\end{align*}
By using discrete Gronwall's inequality (e.g. Chapter XIV, Theorem 1 and Remark 1,2 in \cite{MiPeFi}, page 436-437), there exists $C_{4}>0$ such that
\begin{align*}
\sup_{\theta \in [0,\theta_{0}]}
\sup_{\ell=0,\ldots, n}
\bE[|e^{(n,\theta)}(\ell)|^{2}]
&\leq
C_{4} \Delta t.
\end{align*}
This concludes \eqref{main_3_0}.

By using \eqref{eq_914_0} and induction, we have
\begin{align*}
|e^{(n,\theta)}(\ell)|^{2}
&\leq
L_{b,\sigma}\sum_{m=0}^{\ell-1}|e^{(n,\theta)}(m)|^{2} \Delta t
+
2\sum_{m=0}^{\ell-1}\langle e^{(n,\theta)}(m), R^{(n,\theta)}(\sigma,m)+r^{(n)}(m)\rangle
\\&\quad+
\Big\{1+\frac{\theta_{0}^{2}}{1-2\theta_{0}}\Big\}
\sum_{m=0}^{\ell-1}
|R^{(n,\theta)}(\sigma,b,m)+r^{(n)}(m)|^{2}
\\&\quad+
(1-\theta)^{2}\sum_{m=0}^{\ell-1}\big\{|R^{(n,\theta)}(g,m)|^{2}-|R^{(n,\theta)}(g,m+1)|^{2}\big\}
\\&\leq
(1+L_{b,\sigma})\sum_{m=0}^{\ell-1}|e^{(n,\theta)}(m)|^{2} \Delta t
+
2\Big|\sum_{m=0}^{\ell-1}\langle e^{(n,\theta)}(m), R^{(n,\theta)}(\sigma,m)+r^{(n)}(\sigma,m)\rangle\Big|
\\&\quad+
(\Delta t)^{-1}
\sum_{m=0}^{\ell-1}
|r^{(n)}(b,m)+r^{(n)}(g,m)|^{2}
\\&\quad+
2\Big\{1+\frac{\theta_{0}^{2}}{1-2\theta_{0}}\Big\}
\sum_{m=0}^{\ell-1}
\Big\{
|R^{(n,\theta)}(\sigma,b,m)|^{2}+|r^{(n)}(m)|^{2}
\Big\}.
\end{align*}
Therefore, by taking the supremum with respect to $\ell=1,\ldots,n$ and using \eqref{ass: main_3}, \eqref{main_3_0}, \eqref{R_esti_0},\eqref{R_esti_1}, \eqref{R_esti_2} and \eqref{theta scheme: eq9}, there exists $C_{5}>0$ such that
\begin{align*}
\bE\Big[
\sup_{\ell=1,\ldots,n}
|e^{(n,\theta)}(\ell)|^{2}
\Big]
&\leq
C_{5}\Delta t
+
2
\bE\Big[
\sup_{\ell=1,\ldots,n}
\Big|\sum_{m=0}^{\ell-1}\langle e^{(n,\theta)}(m), R^{(n,\theta)}(\sigma,m)+r^{(n)}(\sigma,m)\rangle\Big|
\Big].
\end{align*}
By the definitions of $R^{(n,\theta)}(\sigma,m)$ and $r^{(n)}(\sigma,m)$, it holds that
\begin{align*}
\sum_{m=0}^{\ell-1}\langle e^{(n,\theta)}(m), R^{(n,\theta)}(\sigma,m)+r^{(n)}(\sigma,m)\rangle
=
M^{(n,\theta)}(t_{\ell}),
\end{align*}
where $M^{(n,\theta)}$ is a martingale defined by
\begin{align*}
M^{(n,\theta)}(t)
:=
\int_{0}^{t}
\langle
V(\eta(s))-V^{(n,\theta)}(\eta(s)),
\{\sigma(s,V(s))-\sigma(\eta(s),V^{(n,\theta)}(\eta(s)))\}\rd B(s)
\rangle.
\end{align*}
Here $\eta$ is defined by $\eta(s):=t_{m}$ if $s \in [t_{m},t_{m+1})$.
By using Burkholder--Davis--Gundy's inequality and Young's inequality, there exists $C_{6}>0$ such that
\begin{align*}
\bE\Big[\sup_{\ell=1,\ldots,n}|M^{(n,\theta)}(t_{\ell})|\Big]
&\leq
\bE\Big[\sup_{t \in [0,T]}|M^{(n,\theta)}(t)|\Big]
\\&\leq
C_{6}
\bE\Big[
\Big(
\int_{0}^{T}
|V(\eta(s))-V^{(n,\theta)}(\eta(s))|^{2}
|\sigma(s,V(s))-\sigma(\eta(s),V^{(n,\theta)}(\eta(s)))|^{2}
\rd s
\Big)^{1/2}
\Big]
\\&\leq
C_{6}
\bE\Big[
\sup_{\ell=1,\ldots,n}
|e^{(n,\theta)}(\ell)|
\Big(
\int_{0}^{T}
|\sigma(s,V(s))-\sigma(\eta(s),V^{(n,\theta)}(\eta(s)))|^{2}
\rd s
\Big)^{1/2}
\Big]
\\&\leq
\frac{1}{4}
\bE\Big[
\sup_{\ell=1,\ldots,n}
|e^{(n,\theta)}(\ell)|^{2}
\Big]
+
C_{6}^{2}
\bE\Big[
\int_{0}^{T}
|\sigma(s,V(s))-\sigma(\eta(s),V^{(n,\theta)}(\eta(s)))|^{2}
\rd s
\Big].
\end{align*}
Hence, we obtain
\begin{align*}
\bE\Big[
\sup_{\ell=1,\ldots,n}
|e^{(n,\theta)}(\ell)|^{2}
\Big]
\leq
2C_{5}\Delta t
+
4C_{6}^{2}
\bE\Big[
\int_{0}^{T}
|\sigma(s,V(s))-\sigma(\eta(s),V^{(n,\theta)}(\eta(s)))|^{2}
\rd s
\Big].
\end{align*}
Finally, by using Assumption \ref{Ass_1} (ii), Assumption \ref{Ass_2} (iii) and \eqref{main_3_0}, there exists $C_{7}>0$ such that
\begin{align*}
\sup_{\theta \in [0,\theta_{0}]}
\bE\Big[\sup_{\ell=0,\ldots,n}|e^{(n,\theta)}(\ell)|^{2}\Big]
\leq
\frac{C_{7}}{\sqrt{n}}.
\end{align*}
This concludes the assertion.
\end{proof}

\section{Numerical schemes for non-colliding particle systems}\label{Sec_4}

Throughout this section, we always suppose that Assumption \ref{Ass_1} holds.
Additionally, we assume that the functions $b$, $\sigma$, and $k$ are $1/2$--H\"older continuous in time, that is,
\begin{align}\label{def: k 1/2}
\begin{split}
[f]_{1/2}
&:=
\sup_{x\in \bW, ~t,s \in [0,T],~t \neq s}
\frac{|f(t,x)-f(s,x)|}{|t-s|^{1/2}}
<\infty,~f \in \{b,\sigma\},\\
[k]_{1/2}
&:=
\sup_{\alpha \in R_{+}, ~t,s \in [0,T],~t \neq s}
\frac{|k(t,\alpha)-k(s,\alpha)|}{|t-s|^{1/2}}
<\infty.
\end{split}
\end{align}

\subsection{$\theta$-Euler--Maruyama scheme}

Let $x \in \bR^d$ and $c:R \to (0,\infty)$ be a measurable function.
Then from Lemma 4.1 in \cite{NT24+}, there exists a unique solution $y$ in the Weyl chamber $\bW$ of the following equation:
\begin{equation} \label{eqn_xi}
y = x + \sum_{\alpha \in R_{+}} \frac{c(\alpha)}{\langle \alpha, y \rangle} \alpha.
\end{equation}
Building on this fact, for any $\theta \in [0,1/2)$, we define a $\theta$-Euler--Maruyama scheme $X^{(n,\theta)} = (X^{(n,\theta)}(t_{\ell}))_{\ell=0}^{n}$ for $X$ as follows:
we set $X^{(n,\theta)}(0):=X(0)=\xi \in \bW$, and for each $\ell=0,\ldots,n-1$, we define $X^{(n,\theta)}(t_{\ell+1})$ as the unique solution in $\bW$ of the following equation:
\begin{align*}
X^{(n,\theta)}(t_{\ell+1})
&=
X^{(n,\theta)}(t_{\ell})
+
\sigma(t_{\ell},X^{(n,\theta)}(t_{\ell}))
\Delta B_{\ell}
+
b(t_{\ell},X^{(n,\theta)}(t_{\ell}))
\Delta t
\\&\quad+
\theta
\sum_{\alpha \in R_{+}}
\frac{k(t_{\ell},\alpha)}{\langle \alpha, X^{(n,\theta)}(t_{\ell}) \rangle}
\alpha
\Delta t
+
(1-\theta)
\sum_{\alpha \in R_{+}}
\frac{k(t_{\ell+1},\alpha)}{\langle \alpha, X^{(n,\theta)}(t_{\ell+1}) \rangle}
\alpha
\Delta t.
\end{align*}

We obtain the following result concerning the rate of strong convergence.

\begin{Thm}\label{main_1}
Assume that $p_{*} > 6$.
Then for any $\theta_{0} \in [0,1/2)$, there exists $C_{\theta_{0}}>0$ such that for any $n \in \bN$,
\begin{equation*}
\sup_{\theta \in [0,\theta_{0}]}
\bE
\Big[
\sup_{\ell =1,\ldots,n}
\Big|
X(t_{\ell})
-
X^{(n,\theta)}(t_{\ell})
\Big|^{2}
\Big]^{1/2}
\leq
\frac{C_{\theta_{0}}}{\sqrt{n}}.
\end{equation*}
\end{Thm}

\begin{proof}
By Lemma \ref{Lem_Kol_0} and  Theorem \ref{main_3}, we need to prove that there exists $C>0$ such that
\begin{equation*}
\sup_{\ell=1,\ldots,n}
\bE[|r^{(n)}(f_{k},\ell)|^{2}]
\leq
C(\Delta t)^{3}.
\end{equation*}
By the definition of $r^{(n)}(f_{k},\ell)$ and using Jensen's inequality, we have
\begin{align}\label{main_1: eq1}
\begin{split}
|r^{(n)}(f_{k},\ell)|^{2}
&\leq
2\theta^{2}
\Delta t
\int_{t_{\ell}}^{t_{\ell+1}}
\big|
f_{k}(s,X(s))
-
f_{k}(t_{\ell},X(t_{\ell}))
\big|^{2}
\rd s
\\&\quad+
2(1-\theta)^{2}
\Delta t
\int_{t_{\ell}}^{t_{\ell+1}}
\big|
f_{k}(s,X(s))
-
f_{k}(t_{\ell+1},X(t_{\ell+1}))
\big|^{2}
\rd s.
\end{split}
\end{align}
By using Assumption (iii), \eqref{def: k 1/2} and Cauchy--Schwarz inequality, for any $s,t \in [0,T]$ with $s \neq  t$ and $x,y \in \bW$, we have
\begin{align*}
|f_{k}(s,x)-f_{k}(t,y)|^{2}
&
\leq
2|f_{k}(s,x)-f_{k}(t,x)|^{2}
+
|f_{k}(t,x)-f_{k}(t,y)|^{2}
\\&\leq
2[k]_{1/2}^{2}|R_{+}|
|t-s|
\sum_{\alpha \in R_{+}}
\frac{|\alpha|^{2}}{\langle \alpha,x \rangle^{2}}
+
2|x-y|^{2}
\sum_{\alpha \in R_{+}}
\|k(\cdot, \alpha)\|^2_{\infty}|R_{+}|
\frac{|\alpha|^{4}}{\langle \alpha,x \rangle^{2} \langle \alpha,y \rangle^{2}}.
\end{align*}
Therefore, by \eqref{main_1: eq1}, H\"older's inequality, Theorem \ref{Thm_negative} and Lemma \ref{Lem_Kol_0}, there exist $C_{1},C_{2}>0$ such that
\begin{align*}
&\bE[|r^{(n)}(f_{k},\ell)|^{2}]
\\&\leq
C_{1}
(\Delta t)^{3}
\sup_{\alpha \in R_{+}}
\sup_{t \in [0,T]}
\bE[\langle \alpha, X(t) \rangle^{-2}]
\\&\quad+
C_{1}
\Delta t
\sup_{\alpha \in R_{+}}
\sup_{t \in [0,T]}
\bE[\langle \alpha, X(t) \rangle^{-6}]^{1/3}
\int_{t_{\ell}}^{t_{\ell+1}}
\big\{
\bE[|X(s)-X(t_{\ell})|^{6}]^{1/3}
+
\bE[|X(s)-X(t_{\ell+1})|^{6}]^{1/3}
\big\}
\rd s
\\&\leq
C_{2}(\Delta t)^{3}.
\end{align*}
The proof is thus completed.
\end{proof}

\subsection{Truncated Euler--Maruyama  $\theta$-scheme}
One drawback of the $\theta$-Euler--Maruyama scheme is the requirement to solve the nonlinear equation \eqref{eqn_xi}, which can be computationally intensive. In this section, we introduce the truncated Euler--Maruyama $\theta$-scheme, which provides improved numerical efficiency. The main idea is to apply the $\theta$-scheme to the system \eqref{SDE_app} rather than the original system \eqref{SDE_1}.

Let $x \in \bR^d$.
Then, for $h \in (0,\varepsilon^{2}/L_{k})$, where $L_{k}$ is defined in \eqref{def_L},  the  equation
\begin{equation} \label{eqn_xi_1}
y=x + hf_{k,\varepsilon}(t,y)
\end{equation}
has a unique solution in $\bR^{d}$.
Moreover, we can construct the unique solution $y$ of \eqref{eqn_xi_1} as follows.
We set $F_{k,\varepsilon}(t,y):=x+hf_{k,\varepsilon}(t,y)$.
Let $y_{0}=x$ and $y_{n}:=F_{k,\varepsilon}(t,y_{n-1})$.
Then by \eqref{f_1}, we have for any $y,y' \in \bR^{d}$, $|F_{k,\varepsilon}(t,y)-F_{k,\varepsilon}(t,y')| \leq L_{k}\varepsilon^{-2}h|y-y'|$.
Since $L_{k}\varepsilon^{-2}h<1$, $y_{n}$ converges to the unique fixed point $y$ of $y=F_{k,\varepsilon}(t,y)=x+f_{k,\varepsilon}(t,y)$.
In particular, it holds that $|y-y_{n}| \leq \frac{\sum_{\alpha \in R_{+}}k(t,\alpha)|\alpha|}{L_{k}(1-L_{k}\varepsilon^{-2}h)}
\cdot \varepsilon\cdot (L_{k}\varepsilon^{-2}h)^{n}.
$

Let $\varepsilon>0$ and $n >T L_{k}/\varepsilon^{2}$.
Based on the above fact, for $\theta \in [0,1)$, we can define a truncated $\theta$-scheme $X_{\varepsilon}^{(n,\theta)}=(X_{\varepsilon}^{(n,\theta)}(t_{\ell}))_{\ell=0}^{n}$ for  $X_{\varepsilon}$ as follows: $X_{\varepsilon}^{(n,\theta)}(0):=X_{\varepsilon}(0)=\xi \in \bW$, and for each $\ell=0,\ldots,n-1$, $X_{\varepsilon}^{(n,\theta)}(t_{\ell+1})$ is the unique solution in $\bR^{d}$ of the following equation:
\begin{align*}
X_{\varepsilon}^{(n,\theta)}(t_{\ell+1})
&=
X_{\varepsilon}^{(n,\theta)}(t_{\ell})
+
\sigma(t_{\ell},X_{\varepsilon}^{(n,\theta)}(t_{\ell}))
\Delta B_{\ell}
+
b(t_{\ell},X_{\varepsilon}^{(n,\theta)}(t_{\ell}))
\Delta t
\\&\quad+
\theta
\sum_{\alpha \in R_{+}}
\frac{k(t_{\ell},\alpha)}{\langle \alpha, X_{\varepsilon}^{(n,\theta)}(t_{\ell})\rangle \vee \varepsilon}
\alpha
\Delta t
+
(1-\theta)
\sum_{\alpha \in R_{+}}
\frac{k(t_{\ell+1},\alpha)}{\langle \alpha, X_{\varepsilon}^{(n,\theta)}(t_{\ell+1}) \rangle \vee \varepsilon}
\alpha
\Delta t.
\end{align*}

We obtain the following result concerning the rate of strong convergence.

\begin{Thm}\label{main_2}
Let $c>1$ and let $\varepsilon_n = c\sqrt{L_k \Delta t}$. 
Assume that $p_{*} > 8$. 
Then for any $\theta_0 \in (0, 1/2)$,
there exists $C_{\theta_{0}}>0$ such that for all $n > TL_k (\min_{\alpha \in R_+}\langle \alpha, \xi\rangle)^{-2}$,
\begin{equation*}
\sup_{\theta \in [0,\theta_0]}
\bE
\Big[
\sup_{\ell =1,\ldots,n}
\Big|
X(t_{\ell})
-
X_{\varepsilon_n}^{(n,\theta)}(t_{\ell})
\Big|^{2}
\Big]^{1/2}
\leq
\frac{C_{\theta_{0}}}{\sqrt{n}}.
\end{equation*}
\end{Thm}

We first prove that SDE $X_{\varepsilon}$ defined in \eqref{SDE_app} approximates $X$ in $L^{p}$ sense.

\begin{Lem}\label{Lem_app_X}
Suppose that $p_{*} >4$.
Then for any $p \in [2,p_{*}/2)$, there exists $C_{p}>0$ such that for all $\varepsilon>0$ it holds that
\begin{align*}
\bE
\Big[
\sup_{0 \leq t \leq T}
\Big|
X(t)
-
X_{\varepsilon}(t)
\Big|^{p}
\Big]^{1/p}
\leq
C_{p}\varepsilon.
\end{align*}
\end{Lem}
\begin{proof}
By using It\^o's formula, Assumption \ref{Ass_1}, items \ref{Ass_1: i}, \ref{Ass_1: ii}, and estimate \eqref{f_2}, we obtain that
\begin{align*}
\left|
X(t)
-
X_{\varepsilon}(t)
\right|^{2}
&\leq
\int_{0}^{t}
\langle X(s)-X_{\varepsilon}(s), \{\sigma(s,X(s))-\sigma(s,X_{\varepsilon}(s))\} \rd B(s) \rangle\\
&\quad
+
2\int_{0}^{t}
\langle X(s)-X_{\varepsilon}(s), b(s,X(s))-b(s,X_{\varepsilon}(s)) \rangle
\rd s\\
&\quad
+2\int_{0}^{t}
\langle X(s)-X_{\varepsilon}(s), f_{k}(s,X(s))-f_{k,\varepsilon}(s,X(s)) \rangle
\rd s \\
&\quad+\int_{0}^{t} |X(s)-X_{\varepsilon}(s)|^{2} \rd s.
\end{align*}
Since $X(t) \in \bW$, a.s., \eqref{f_3} and Burkholder--Davis--Gundy's inequality imply that there exists a positive constant $C_{p}$ such that for all $t \in [0,T]$ it holds that
\begin{align*}
\bE\Big[
\sup_{0 \leq u \leq t}
\Big|
X(u)
-
X_{\varepsilon}(u)
\Big|^{p}
\Big]
&\leq
C_{p}
\int_{0}^{t}
\bE\Big[
\sup_{0 \leq u \leq s}
\Big|
X(u)
-
X_{\varepsilon}(u)
\Big|^{p}
\Big]
\rd s
\\&\quad+
C_{p}
\varepsilon^{p}
\sum_{\alpha \in R_{+}}
\int_{0}^{T}
\bE[\langle \alpha,X(s) \rangle^{-2p}]
\rd s.
\end{align*}
By using Theorem \ref{Thm_negative} and  Gronwall's inequality, we obtain the desired result.
\end{proof}

For $\varepsilon>0$, we define a stopping times $\tau_{\varepsilon}^{\alpha}$ by: 
$\tau_{\varepsilon}^{\alpha}(w) :=\inf\{s \in (0,T]~|~\langle \alpha, X(s, w) \rangle=\varepsilon\}$ for $\alpha \in R_{+}$. 
Let set $\tau_{\varepsilon}:=\min_{\alpha \in R_{+}} \tau_{\varepsilon}^{\alpha}$.
Then we estimate the probability of this event $\tau_{\varepsilon} \leq t_{\ell}$ for each $\ell=1,\ldots,n$.

\begin{Lem}\label{Lem_theta_0}
Assume that $p_{*}>2$.
Then for any $p \in (0,p_{*}-2)$, there exists $C_{p}>0$ such that
\begin{align*}
\sup_{\ell=1,\ldots,n}
\bP(\tau_{\varepsilon} \leq t_{\ell})
\leq
\sup_{\ell=1,\ldots,n}
\sum_{\alpha \in R_{+}}
\bP\Big(
\min_{s \in [0,t_{\ell}]}
\langle \alpha, X(s) \rangle \leq \varepsilon
\Big)
\leq
C_{p}\varepsilon^{p}.
\end{align*}
\end{Lem}
\begin{proof}
By the definition of the stopping times $\tau_{\varepsilon_n}^{\alpha}$, $\alpha \in R_{+}$ and Theorem \ref{Thm_negative}, it follows from Jensen's inequality and Markov's inequality that
\begin{align*}
\bP(\tau_{\varepsilon}\leq t_{\ell})
&=
\bP\Big(
\bigcup_{\alpha \in R_{+}}
\Big\{
\min_{s \in [0,t_{\ell}]}
\langle \alpha, X(s) \rangle \leq \varepsilon
\Big\}
\Big)
\leq
\sum_{\alpha \in R_{+}}
\bP\Big(
\min_{s \in [0,t_{\ell}]}
\langle \alpha, X(s) \rangle \leq \varepsilon
\Big)
\\&\leq
\varepsilon^{p}
\sum_{\alpha \in R_{+}}
\bE\Big[
\sup_{s \in [0,T]}
\langle \alpha, X(s) \rangle^{-p}
\Big]
\leq
C_{p}
\varepsilon^{p},
\end{align*}
for some $C_{p}>0$.
This concludes the assertion.
\end{proof}

We are now ready to prove Theorem \ref{main_2}.


\begin{proof}[Proof of Theorem \ref{main_2}]
By Theorem \ref{main_3}, Lemma \ref{Lem_Kol_0}, and Lemma \ref{Lem_app_X}, we need to prove that there exists $C>0$ such that
\begin{equation}\label{need to prove: main 2: eq1}
\sup_{\ell=1,\ldots,n}
\bE[|r^{(n)}(f_{k,\varepsilon_n},\ell)|^{2}]
\leq
C(\Delta t)^{3}.
\end{equation}
By the definition of $r^{(n)}(f_{k,\varepsilon_{n},\ell})$, we have
\begin{align}\label{eq_r_fn}
\begin{split}
\bE[|r^{(n)}(f_{k,\varepsilon_{n}},\ell)|^{2}]
&\leq
2\Delta t
\int_{t_{\ell}}^{t_{\ell+1}}
\bE\Big[
\big|f_{k,\varepsilon_{n}}(s,X_{\varepsilon_n}(s))-f_{k,\varepsilon_{n}}(t_{\ell},X_{\varepsilon_n}(t_{\ell}))\big|^{2}
\\&\hspace{3cm}
+
\big|f_{k,\varepsilon_{n}}(s,X_{\varepsilon_n}(s))-f_{k,\varepsilon_{n}}(t_{\ell+1},X_{\varepsilon_n}(t_{\ell+1}))\big|^{2}
\Big]
\rd s.
\end{split}
\end{align}
Note that by the definition of $f_{k,\varepsilon_{n}}$, for any $s,t \in [t_{\ell},t_{\ell+1}]$ with $s < t$ and $x,y \in \bR^{d}$ it holds that
\begin{align}\label{eq_thm_0}
&|f_{k,\varepsilon_n}(s,x)-f_{k,\varepsilon_n}(t,y)|^{2}
\leq
2|f_{k,\varepsilon_n}(s,x)-f_{k,\varepsilon_n}(t,x)|^{2}
+
2|f_{k,\varepsilon_n}(t,x)-f_{k,\varepsilon_n}(t,y)|^{2}
\notag
\\&\leq
2\Big|
\sum_{\alpha \in R_{+}}
\{k(s,\alpha)-k(t,\alpha)\}g_{\varepsilon_n}(\langle \alpha,x \rangle)
\alpha
\Big|^{2}
+
2\Big|
\sum_{\alpha \in R_{+}}
k(s,\alpha)
\{g_{\varepsilon_n}(\langle \alpha,x \rangle)-g_{\varepsilon_n}(\langle \alpha,y \rangle)\}
\alpha
\Big|^{2}
\notag
\\&\leq
2[k]_{1/2}^{2}|R_{+}|
(t-s)
\sum_{\alpha \in R_{+}}
|\alpha|^{2}
g_{\varepsilon_n}(\langle \alpha,x \rangle)^{2}
\notag
\\&\quad+
2|R_{+}| \sum_{\alpha \in R_{+}} \|k(\cdot, \alpha)\|^2_{\infty}
| \alpha|^{2}| 
g_{\varepsilon_n}(\langle \alpha,x \rangle)-g_{\varepsilon_n}(\langle \alpha,y \rangle)|^{2}.
\end{align}
We consider the estimate for the right hand side of \eqref{eq_r_fn} on $\tau_{\varepsilon_n}>t_{\ell+1}$ and $\tau_{\varepsilon_n}\leq t_{\ell+1}$, respectively.
First, we first consider  $\omega \in \{\tau_{\varepsilon_n}>t_{\ell+1}\}$.
Then
$\min_{\alpha \in R_{+}} \min_{s \in [0,t_{\ell+1}]} \langle \alpha,X(s,\omega) \rangle >\varepsilon_n$.
Hence, it holds that for all $t \in [0,t_{\ell+1}]$, $\omega\in\Omega$ with $\tau_{\varepsilon_n}(\omega)>t_{\ell+1}$,
\begin{align*}
X(t,\omega)=X_{\varepsilon_n}(t,\omega)
\text{ and }
g_{\varepsilon_n}(\langle \alpha,X_{\varepsilon_n}(t,\omega) \rangle)=\langle \alpha,X(t,\omega) \rangle^{-1}.
\end{align*}
Therefore, by using Lemma \ref{Thm_negative} and Lemma \ref{Lem_Kol_0}, there exist $C_{1},C_{2}>0$ such that for any $s,t \in [t_{\ell}, t_{\ell+1}]$ with $s<t$
\begin{align}\label{eq_thm_1}
\bE\Big[
g_{\varepsilon_n}(\langle \alpha, X_{\varepsilon_n}(t) \rangle)^{2}
\1_{\{\tau_{\varepsilon_n}>t_{\ell+1} \}}
\Big]
&=
\bE\Big[
\langle \alpha, X(t) \rangle^{-2}
\1_{\{\tau_{\varepsilon_n}>t_{\ell+1}\}}
\Big]
\leq
\bE\Big[
\langle \alpha, X(t) \rangle^{-2}
\Big]
\leq
C_{1},
\end{align}
and
\begin{align}\label{eq_thm_2}
&
\bE\Big[
\Big|
g_{\varepsilon_n}(\langle \alpha, X_{\varepsilon_n}(s) \rangle)
-
g_{\varepsilon_n}(\langle \alpha, X_{\varepsilon_n}(t) \rangle)
\Big|^{2}
\1_{\{\tau_{\varepsilon_n}>t_{\ell+1}\}}
\Big]
\notag
\\&\leq
\bE\Big[
\frac
{|\alpha|^2|X(t)-X(s)|^2}
{\langle \alpha, X(s) \rangle^{2}\langle \alpha, X(t) \rangle^{2}}
\Big]
\notag
\\&\leq
|\alpha|^2 
\bE[|X(s)-X(t)|^{6}]^{1/3}
\bE[\langle \alpha, X(s) \rangle^{-6}]^{1/3}
\bE[\langle \alpha, X(t) \rangle^{-6}]^{1/3}
\notag
\\&\leq
C_{2} \Delta t.
\end{align}


We next consider the case $\tau_{\varepsilon_n}\leq t_{\ell+1}$.
By using Lemma \ref{Lem_theta_0} with $p=6$, there exists $C_{3}>0$ such that for any $t \in [t_{\ell}, t_{\ell+1}]$,
\begin{align}\label{eq_thm_3}
\bE\Big[
g_{\varepsilon_n}(\langle \alpha, X_{\varepsilon_n}(t) \rangle)^{2}
\1_{\{\tau_{\varepsilon_n} \leq t_{\ell+1}\}}
\Big]
&\leq 
\varepsilon_n^{-2}
\bP(\tau_{\varepsilon_n} \leq t_{\ell+1})
\leq
C_{3}\varepsilon_n^{4}.
\end{align}
Moreover, by using Lipschitz continuity of $f_{k,\varepsilon_n}$ \eqref{f_1} and H\"older's inequality, there exists $C_{4}>0$ such that for any $s,t \in [t_{\ell}, t_{\ell+1}]$ with $s<t$,
\begin{align*}
&
\bE\Big[
\Big|
g_{\varepsilon_n}(\langle \alpha, X_{\varepsilon_n}(s) \rangle)
-
g_{\varepsilon_n}(\langle \alpha, X_{\varepsilon_n}(t) \rangle)
\Big|^{2}
\1_{\{\tau_{\varepsilon_n} \leq t_{\ell+1}\}}
\Big]
\\&\leq
C_{4}
\varepsilon_n^{-4}
\bE\Big[
\Big|
\langle \alpha, X_{\varepsilon_n}(s) \rangle
-
\langle \alpha, X_{\varepsilon_n}(t) \rangle
\Big|^{6}
\Big]^{1/3}
\bP(\tau_{\varepsilon_n}\leq t_{\ell+1})^{2/3}
\\&\leq
C_{3}^{2/3}C_{4}|\alpha|^{2}
\bE\Big[
\Big|
X_{\varepsilon_n}(s)
-
X_{\varepsilon_n}(t)
\Big|^{6}
\Big]^{1/3}.
\end{align*}
Since $\sup_{(t,x) \in [0,T] \times \bR^{d}}|f_{k,\varepsilon_n}(t,x)|^{2} \leq \sum_{\alpha \in R_{+}}\|k(\cdot, \alpha)\|_{\infty}^{2} |\alpha|^{2} \varepsilon_n^{-2}$ and $\varepsilon_n^{2}=c^2 L_{k} \Delta t$, there exists $C_5>0$ such that for any $s,t\in [t_\ell,t_{\ell+1}]$, $s < t$, 
\begin{align*}
&\bE\Big[
\Big|
X_{\varepsilon_n}(s)
-
X_{\varepsilon_n}(t)
\Big|^{6}
\Big]
\\&\leq
3^{5}
\bE\Big[
\Big|
\int_{s}^{t}
\sigma(u,X_{\varepsilon_n}(u))
\rd B(u)
\Big|^{6}
+
(\Delta t)^{5}
\int_{s}^{t}
\Big\{
|b(u,X_{\varepsilon_n}(u))|^{6}
+
|f_{k,\varepsilon_n}(u,X_{\varepsilon_n}(u))|^{6}
\Big\}
\rd u
\Big]
\\&\leq
C_{5}(\Delta t)^{3}.
\end{align*}
Therefore we have
\begin{align}\label{eq_thm_4}
\bE\Big[
\Big|
g_{\varepsilon_n}(\langle \alpha, X_{\varepsilon_n}(s) \rangle)
-
g_{\varepsilon_n}(\langle \alpha, X_{\varepsilon_n}(t) \rangle)
\Big|^{2}
\1_{\{t_{\ell+1} \geq \tau_{\varepsilon_n}\}}
\Big]
&\leq
C_{6}
\Delta t.
\end{align}

Combining  \eqref{eq_r_fn}, \eqref{eq_thm_0}, \eqref{eq_thm_1}, \eqref{eq_thm_2}, \eqref{eq_thm_3} and \eqref{eq_thm_4}, we obtain \eqref{need to prove: main 2: eq1}.
The proof is thus completed.
\end{proof}
Let $c>1$ and let $\varepsilon_n = c\sqrt{L_k \Delta t}$.
We finally prove that the probability of the event $X^{(n,\theta)}_{\varepsilon_{n}}(t_{\ell}) \in \bW$ for all $\ell=1,\ldots,n$, converges to $1$ as $n \to \infty$.
\begin{Thm}\label{main_5}
Let $c>1$ and let $\varepsilon_n = c\sqrt{L_k \Delta t}$.
Assume that $p_{*}>8$ and let $\theta_{0} \in [0,1/2)$.
Then for any $\gamma \in [3/4,1-2/p_{*})$, there exists $C_{\gamma,\theta_{0}}>0$ such that for any $n > TL_k (\min_{\alpha \in R_+}\langle \alpha, \xi\rangle)^{-2}$ it holds that
\begin{align*}
\inf_{\theta \in [0,\theta_{0}]}
\bP\big(
X^{(n,\theta)}_{\varepsilon_{n}}(t_{\ell}) \in \bW,~
\mathrm{for~all~}
\ell=1,\ldots,n
\big)
\geq
1-\frac{C_{\gamma,\theta_{0}}}{n^{\gamma}}.
\end{align*}
\end{Thm}

\begin{proof}
For $\theta \in [0,1/2)$, we set a stopping time $\tau_{\varepsilon_{n}}^{(n,\theta)}$ by
\begin{align*}
\tau_{\varepsilon_{n}}^{(n,\theta)}:=\min_{\alpha \in R_{+}}\inf\Big\{t_{\ell}~\Big|~\ell=1,\ldots,n,~\langle \alpha,X^{(n,\theta)}_{\varepsilon_{n}}(t_{\ell}) \rangle \leq \varepsilon_{n}\Big\}.
\end{align*}
Then we first prove that for any $\gamma \in [3/4,1-2/p_{*})$, there exists $C_{\gamma,\theta_{0}}>0$ such that for any $n > TL_k (\min_{\alpha \in R_+}\langle \alpha, \xi\rangle)^{-2}$, it holds that
\begin{align}\label{eq_5_0}
\sup_{\theta \in [0,\theta_{0}]}
\bP(\tau_{\varepsilon_{n}}^{(n,\theta)} \leq T)
\leq
\frac{C_{\gamma,\theta_{0}}}{n^{\gamma}}.
\end{align}

For $\delta>0$, we define the event $\Omega_{\delta,n,\theta}$ by
\begin{align*}
\Omega_{\delta,n,\theta}
:=
\Big\{
\sup_{\ell =1,\ldots,n}
\Big|
X(t_{\ell})
-
X_{\varepsilon_n}^{(n,\theta)}(t_{\ell})
\Big|
\geq \delta
\Big\}.
\end{align*}
Then we have
\begin{align*}
\bP(\tau_{\varepsilon_{n}}^{(n,\theta)} \leq T)
&=
\bP(\{\tau_{\varepsilon_{n}}^{(n,\theta)} \leq T\}\cap \Omega_{\delta,n,\theta})
+
\bP(\{\tau_{\varepsilon_{n}}^{(n,\theta)} \leq T\}\cap \Omega_{\delta,n,\theta}^{c})
\\&\leq
\bP(\Omega_{\delta,n,\theta})
+
\bP(\{\tau_{\varepsilon_{n}}^{(n,\theta)} \leq T\}\cap \Omega_{\delta,n,\theta}^{c}).
\end{align*}
By Theorem \ref{main_2}, it holds that
\begin{align*}
\sup_{\theta \in [0,\theta_0]}
\bP(\Omega_{\delta,n,\theta})
&\leq
\frac{1}{\delta^{2}}
\sup_{\theta \in [0,\theta_0]}
\bE
\Big[
\sup_{\ell =1,\ldots,n}
\Big|
X(t_{\ell})
-
X_{\varepsilon_n}^{(n,\theta)}(t_{\ell})
\Big|^{2}
\Big]
\leq
\frac{C_{\theta_{0}}}{\delta^{2} n}.
\end{align*}
On the other hand, it holds that
\begin{align*}
\bP(\{\tau_{\varepsilon_{n}}^{(n,\theta)} \leq T\}\cap \Omega_{\delta,n,\theta}^{c})
\leq
\sum_{\alpha \in R_{+}}
\bP\Big(
\Big\{
\min_{\ell =0,\ldots,n}
\langle \alpha, X_{\varepsilon_{n}}^{(n,\theta)}(t_{\ell}) \rangle \leq \varepsilon_{n}
\Big\}
\cap
\Omega_{\delta,n,\theta}^{c}
\Big).
\end{align*}
For each $\alpha \in R_{+}$, on the event $\{\min_{\ell =0,\ldots,n}
\langle \alpha, X_{\varepsilon_{n}}^{(n,\theta)}(t_{\ell}) \rangle \leq \varepsilon_{n}\}\cap \Omega_{\delta,n,\theta}^{c}$, there exists $\ell_{*}=0,\ldots,n$ such that
$
\langle \alpha, X_{\varepsilon_{n}}^{(n,\theta)}(t_{\ell_{*}}) \rangle
=
\min_{\ell =0,\ldots,n}
\langle \alpha, X_{\varepsilon_{n}}^{(n,\theta)}(t_{\ell}) \rangle \leq \varepsilon_{n}
$, and thus we have
\begin{align*}
\min_{\ell =0,\ldots,n}
\langle \alpha, X(t_{\ell}) \rangle
&\leq
\langle \alpha, X(t_{\ell_{*}}) \rangle
\leq
\langle \alpha, X_{\varepsilon_{n}}^{(n,\theta)}(t_{\ell_{*}}) \rangle
+
|\alpha|
\sup_{\ell=0,\ldots,n}
|X(t_{\ell})-X_{\varepsilon_{n}}^{(n,\theta)}(t_{\ell})|
\\&\leq
\varepsilon_{n}+\max_{\beta \in R_{+}}|\beta| \delta.
\end{align*}
Let $\gamma \in [3/4,1-2/p_{*})$, then $2\gamma/(1-\gamma) \in [6,p_{*}-2)$.
Then by Lemma \ref{Lem_theta_0}, we obtain
\begin{align*}
\sup_{\theta \in [0,\theta_{0}]}
\bP(\{\tau_{\varepsilon_{n}}^{(n,\theta)} \leq T\}\cap \Omega_{\delta,n,\theta}^{c})
&\leq
\sum_{\alpha \in R_{+}}
\bP\Big(\min_{\ell =0,\ldots,n}\langle \alpha, X(t_{\ell}) \rangle\leq \varepsilon_{n}+\max_{\beta \in R_{+}}|\beta|\delta\Big)
\\&\leq
C_{\gamma}
\Big\{
\frac{1}{n^{\gamma/(1-\gamma)}}
+
\delta^{2\gamma/(1-\gamma)}
\Big\},
\end{align*}
for some $C_{\gamma}>0$.
Therefore, we have
\begin{align*}
\sup_{\theta \in [0,\theta_{0}]}
\bP(\tau_{\varepsilon_{n}}^{(n,\theta)} \leq T)
\leq
\frac{C_{\theta_{0}}}{\delta^{2}n}
+
C_{\gamma}
\Big\{
\frac{1}{n^{\gamma/(1-\gamma)}}
+
\delta^{2\gamma/(1-\gamma)}
\Big\}.
\end{align*}
By choosing $\delta=n^{-(1-\gamma)/2}$, we obtain \eqref{eq_5_0} with some constant $C_{\gamma,\theta_{0}}>0$.
By using \eqref{eq_5_0}, it holds that
\begin{align*}
&\inf_{\theta \in [0,\theta_{0}]}
\bP\big(
X^{(n,\theta)}_{\varepsilon_{n}}(t_{\ell}) \in \bW,~
\mathrm{for~all~}
\ell=1,\ldots,n
\big)
\\&=
1
-
\sup_{\theta \in [0,\theta_{0}]}
\bP\big(
\mathrm{there~exist~}
\ell=1,\ldots,n
\mathrm{~and~}
\alpha \in R_{+}
~\mathrm{such~that}~
\langle \alpha,X^{(n,\theta)}_{\varepsilon_{n}}(t_{\ell}) \rangle
\leq 0
\big)
\\&\geq
1-\sup_{\theta \in [0,\theta_{0}]}
\bP(\tau_{\varepsilon_{n}}^{(n,\theta)} \leq T)
\geq
1-\frac{C_{\gamma,\theta_{0}}}{n^{\gamma}}.
\end{align*}

This concludes the assertion.
\end{proof}

\section*{Acknowledgements}
The authors are deeply grateful to Professor Yushi Hamaguchi for his valuable comments.
The second author was supported by the NAFOSTED under Grant Number 101.03-2021.36.
The third author was supported by JSPS KAKENHI Grant Numbers 21H00988 and 23K12988.

\end{document}